\documentclass[a4paper,10pt]{article}
\usepackage[utf8]{inputenc}

\usepackage{wrapfig}
\usepackage{amsmath} 
\usepackage{amsthm} 
\usepackage{amssymb} 
\usepackage{enumerate} 
\usepackage{esint} 
\usepackage{pgf,tikz} 
\usetikzlibrary{arrows} 
 \usepackage{yfonts} 
 \usepackage{mathrsfs} 
 \usepackage{graphicx}
\usepackage{caption}
\usepackage{subcaption}
\usepackage{mathtools} 

\definecolor{ffffff}{rgb}{1.0,1.0,1.0}
\definecolor{qqqqff}{rgb}{0.0,0.0,1.0}
\definecolor{ffqqqq}{rgb}{1.0,0.0,0.0}
\definecolor{zzzzqq}{rgb}{0.6,0.6,0.0}

\newcommand*\circled[1]{\tikz[baseline=(char.base)]{
            \node[shape=circle,draw,inner sep=2pt] (char) {#1};}}

\newcommand*\squared[1]{\tikz[baseline=(char.base)]{
            \node[shape=rectangle,draw,inner sep=2.4pt] (char) {#1}; \node[shape=rectangle,draw,inner sep=1pt] (char) {#1};}}

\newcommand{\C}{{\mathbb C}}       
\newcommand{\R}{{\mathbb R}}       
\newcommand{\N}{{\mathbb N}}       
\newcommand{\Z}{{\mathbb Z}}       
\newcommand{\DDD}{{\mathbb D}}
\newcommand{\diam}{{\rm diam}}
\newcommand{\dist}{{\rm dist}}
\newcommand{\Dist}{{\rm D}}
\newcommand{\Sh}{{\mathbf {Sh}}} 
\newcommand{\SH}{{\mathbf {SH}}} 
\newcommand{\real}{{\rm Re \,}}
\newcommand{\imag}{{\rm Im}}
\newcommand{\rf}[1]{{(\ref{#1})}}
\newcommand{\supp}{{\rm supp}}

\newcommand{\Beurling}{{\mathcal B}}
\newcommand{\Cauchy}{{\mathcal C}}

\newcommand{\norm}[1]{{\left\| {#1} \right\|}}

\newtheorem{theorem}{Theorem}
\newtheorem*{theorem*}{Theorem}
\newtheorem{lemma}[theorem]{Lemma}
\newtheorem{claim}[theorem]{Claim}

\newtheorem*{corollary*}{Corollary}
\newtheorem{proposition}[theorem]{Proposition}
\newtheorem{definition}[theorem]{Definition}

\newtheorem{example}[theorem]{Example}
\newtheorem{remark}[theorem]{Remark}

\numberwithin{subsection}{section}
\numberwithin{theorem}{section}
\numberwithin{equation}{section}
\numberwithin{figure}{section}

\usepackage[affil-it]{authblk}

\title{Sobolev regularity of quasiconformal mappings on domains}

\author{Mart\'i Prats
\thanks{MP (Departament de Ma\-te\-m\`a\-ti\-ques, Universitat Aut\`onoma de Bar\-ce\-lo\-na, Catalonia): \texttt{mprats@mat.uab.cat}.}}

\begin{document}
\maketitle
\bibliographystyle{alpha}

\begin{abstract} 
Consider a Lipschitz domain $\Omega$ and a measurable function $\mu$ supported in $\overline\Omega$ with $\left\|{\mu}\right\|_{L^\infty}<1$. Then the derivatives of a quasiconformal solution of the Beltrami equation $\overline{\partial} f =\mu\, \partial f$ inherit the Sobolev regularity $W^{n,p}(\Omega)$ of the Beltrami coefficient $\mu$ as long as $\Omega$ is regular enough. The condition obtained is that the outward unit normal vector $N$ of the boundary of the domain is in the trace space, that is, $N\in B^{n-1/p}_{p,p}(\partial\Omega)$. 
\end{abstract}

\section{Introduction}
Let $\mu\in L^\infty$ supported in a certain ball $B\subset \C$ with $\norm{\mu}_{L^\infty}<1$ and consider $K:=\frac{1+\norm{\mu}_{L^\infty}}{1-\norm{\mu}_{L^\infty}}$. We say that $f$ is a $K$-quasiregular solution to the Beltrami equation
\begin{equation}\label{eqBeltrami}
\overline{\partial} f =\mu\, \partial f
\end{equation}
with Beltrami coefficient $\mu$ if $f\in W^{1,2}_{loc}$, that is, if $f$ and $\nabla f$ are square integrable functions in any compact subset of $\C$, and $\overline{\partial} f (z)=\mu(z) \partial f(z)$ for almost every $z\in\C$. Such a function $f$ is said to be a $K$-quasiconformal mapping if it is a homeomorphism of the complex plane. If, moreover, $f(z)=z+\mathcal{O}(\frac1z)$ as $z\to\infty$, then we say that $f$ is the principal solution to \rf{eqBeltrami}. 

Given a compactly supported Beltrami coefficient $\mu$, the existence and uniqueness of the principal solution is granted by the measurable Riemann mapping Theorem (see \cite[Theorem 5.1.2]{AstalaIwaniecMartin}, for instance). 
A natural question is to what spaces $f$ belongs.  The goal of this paper is to prove the following theorem.
\begin{theorem}\label{theoInvertBeltrami}
Let $n\in \N$, let $\Omega$ be a bounded Lipschitz domain with outward unit normal vector $N$ in $B^{n-1/p}_{p,p}(\partial\Omega)$ for some $2<p<\infty$ and let $\mu\in W^{n,p}(\Omega)$ with $\norm{\mu}_{L^\infty}<1$ and $\supp(\mu)\subset\overline{\Omega}$. Then, the principal solution  $f$ to \rf{eqBeltrami} is in the Sobolev space $W^{n+1,p}(\Omega)$.
\end{theorem}

The principal solution can be given by means of the Cauchy and the Beurling transforms. For $g\in C^\infty_c$ its Cauchy transform is defined as
$$\Cauchy g(z):= \frac{1}{\pi}\int \frac{g(w)}{z-w}dm(w) \mbox{\quad\quad for all }z\in \C,$$
and its Beurling transform, as
$$\Beurling g(z):=\lim_{\varepsilon\to 0} \frac{-1}{\pi}\int_{|w-z|>\varepsilon} \frac{g(w)}{(z-w)^2}dm(w) \mbox{\quad\quad for almost every }z\in \C.$$
The Beurling transform is a bounded operator in $L^p$ for $1<p<\infty$ and for $g\in W^{1,p}(\C)$ we have that $\Beurling (\overline{\partial} g)=\partial g$. Given a ball $B$, the Cauchy transform sends functions in $L^p(B)$ and vanishing in the complement of $B$ to $W^{1,p}(\C)$ when $2<p<\infty$. Furthermore, the operator $I-\mu \Beurling$ is invertible in $L^2$ and, if we call 
\begin{equation*}
h:=(I-\mu \Beurling)^{-1} \mu,
\end{equation*}
then 
\begin{equation*}
f(z)= \Cauchy h(z) + z
\end{equation*}
is the principal solution of \rf{eqBeltrami} with $\overline\partial f= h$ and  $\partial f= \Beurling h +1$.

The key point to prove Theorem \ref{theoInvertBeltrami} is inverting the operator $(I-\mu \Beurling)$ in a suitable space. 
Astala showed in \cite{Astala} that $h\in L^p$ for $1+k<p<1+1/k$ (in fact, since $h$ is also compactly supported, one can say the same for every $1\leq p \leq 1+k$ even though $(I-\mu \Beurling)$ may not be invertible in $L^p$ for that values of $p$, as shown by Astala, Iwaniec and Saksman in \cite{AstalaIwaniecSaksman}). 
Clop et al. in \cite{ClopFaracoMateuOrobitgZhong} and Cruz, Mateu and Orobitg in \cite{CruzMateuOrobitg} proved that if $\mu$ belongs to the Sobolev space $W^{s,p}(\C)$ (in the Bessel potential sense when $s\notin\N$) with $sp>2$ then also $h\in W^{s,p}(\C)$. One also finds some results in the same spirit for the critical case $sp=2$ and the subcritical case $sp<2$ in  \cite{ClopFaracoMateuOrobitgZhong} and \cite{ClopFaracoRuiz}, but here the space to which $h$ belongs is slightly worse than the space to which $\mu$ belongs, that is, either some integrability or some smoothness is lost.

When it comes to dealing with a Lipschitz domain $\Omega$ with $\supp(\mu)\subset\overline\Omega$, Mateu, Orobitg and Verdera showed in \cite{MateuOrobitgVerdera} that, if the parameterizations of the boundary of $\Omega$ are in $C^{1,\varepsilon}$ with $0<\varepsilon<1$, then  for every $0<\sigma<\varepsilon$  one has that 
\begin{equation}\label{eqEpsilon}
\mu\in C^{0,\varepsilon}(\Omega) \implies  h\in C^{0,\sigma}(\Omega).
\end{equation}
Furthermore, the principal solution to \rf{eqBeltrami} is bilipschitz in that case. The authors allow $\Omega$ to be a finite union of disjoint domains with boundaries overlapping in sets of positive length. In \cite{CittiFerrari}, Giovanna Citti and Fausto Ferrari proved that, if one does not allow any overlapping at all, then \rf{eqEpsilon} holds for $\sigma=\varepsilon$. In \cite{CruzMateuOrobitg} the authors study also the Sobolev spaces to conclude that for the same kind of domains, when $0<\sigma<\varepsilon<1$ and $1<p<\infty$ with $\sigma p>2$ one has that
\begin{equation}\label{eqGap}
\mu \in W^{\sigma,p}(\Omega)\implies h\in W^{\sigma,p} (\Omega).
\end{equation}
A key point is proving the boundedness of the Beurling transform in $W^{\sigma,p}(\Omega)$. To do so, the authors note that $\Beurling\chi_\Omega \in W^{\sigma,p}(\Omega)$ by means of some results from \cite{MateuOrobitgVerdera} and then they prove a $T(1)$ theorem that grants the boundedness of $\Beurling$ in $W^{\sigma,p}(\Omega)$ if $\Beurling\chi_\Omega \in W^{\sigma,p}(\Omega)$. The other key point is the invertibility of $I-\mu \Beurling$ in $W^{\sigma,p}(\Omega)$, which is shown using Fredholm theory.

Cruz and Tolsa proved in \cite{CruzTolsa} that for $0<s\leq1$, $1<p<\infty$ with $sp>1$, if the outward unit normal vector $N$ is in the Besov space ${B}^{s-1/p}_{p,p}(\partial\Omega)$ then $\Beurling\chi_\Omega\in W^{s,p}(\Omega)$. This condition is necessary for Lipschitz domains with small Lipschitz constant (see \cite{TolsaSharp}). Moreover, the fact that $N\in B^{s-1/p}_{p,p}(\partial\Omega)$ implies the parameterizations of the boundary of $\Omega$ to be in $B^{s+1-1/p}_{p,p}$ and, for  $sp>2$, the parameterizations are in $C^{1+s-2/p}$ by a well-known embedding theorem. In that situation, one can use the $T(1)$ result in \cite{CruzMateuOrobitg} to deduce the boundedness of the Beurling transform in $W^{s,p}(\Omega)$. However, their result on quasiconformal mappings only allows to infer that  for every $2/p<\sigma<s-2/p$ we have that \rf{eqGap} holds.

Note that Theorem \ref{theoInvertBeltrami} only deals with the natural values of $s$, but the restrictions $\sigma<s-2/p$ and $s<1$ are eliminated.  For $n=1$ the author expects this to be a sharp result in view of \cite{TolsaSharp}. 

In \cite{PratsPlanarDomains} the author proved that the Beurling transform is bounded in $W^{n,p}(\Omega)$, reaching the following result:
\begin{theorem}\label{theoGeometricNaive}
Consider $p>2$, and $n\in \N$ and let $\Omega$ be a bounded Lipschitz domain with  $N\in B^{n-1/p}_{p,p}(\partial\Omega)$. Then, for every $f\in W^{n,p}(\Omega)$ we have that
\begin{equation*}
\norm{\Beurling(\chi_\Omega f)}_{W^{n,p}(\Omega)}\leq C \norm{N}_{B^{n-1/p}_{p,p}(\partial\Omega)}\norm{f}_{W^{n,p}(\Omega)},
\end{equation*}
where $C$ depends on $p$, $n$, $\diam(\Omega)$ and the Lipschitz character of the domain.
\end{theorem}

In this paper we will face the invertibility of $(I-\mu \Beurling)(\chi_\Omega \cdot)$ in $W^{n,p}(\Omega)$. We will follow the scheme that Iwaniec used in \cite{Iwaniec} to show that $I-\mu \Beurling$ is invertible in every $L^p$ for $1<p<\infty$ when $\mu\in VMO$. That is, we will reduce the proof to the compactness of certain commutators. In our context, however, as it also happens in \cite{CruzMateuOrobitg}, we will have to deal with the compactness of the operator $\chi_\Omega \Beurling \left(\chi_{\Omega^c} \Beurling\left(\chi_\Omega\cdot \right)\right)$ as well. Their approach was based on a result in \cite{MateuOrobitgVerdera} that could be useful for the case $W^{\sigma,p}(\Omega)$ with $\sigma<n-2/p$ but it is not sufficiently strong to deal with the endpoint case $W^{n,p}(\Omega)$, so we present here a new approach which entangles some interesting nuances (see Section \ref{secTechnical}).

Let us stress  a crucial step in Iwaniec's scheme. We need to bound not only the Beurling transform but its iterates $\Beurling^m$ or, more precisely, we need the norm of  $\mu^m \Beurling^m(\chi_\Omega\cdot):W^{n,p}(\Omega)\to W^{n,p}(\Omega)$ to be small  for $m$ big enough. Thus, Theorem \ref{theoGeometricNaive} above is too naive, and we need a quantitative version. The reader may expect to find a bound with a polynomial behavior with respect to $m$, but the fact is that the author has not been able to get such an estimate. Instead, we will use an upper bound for the norm with exponential growth on $m$ but the base will be chosen as close to $1$ as desired, as shown in \cite[Theorem 3.15]{PratsPlanarDomains}. This will suffice to prove Theorem \ref{theoInvertBeltrami}.

The plan of the paper is the following. In Section \ref{secPreliminaries} some preliminary assumptions are stated. Subsection \ref{secNotation} explains the notation to be used and recalls some well-known facts. In Subsection \ref{secChainsAndPolynomials} one recalls some tools to be used in the proof of the main result. In Subsection \ref{secSpaces} the definition of the Besov spaces $B^s_{p,p}$ is given along with some well-known facts. Subsection \ref{secOperators} is about some operators related to the Beurling transform, providing a standard notation for the whole article, and recalling the precise results from \cite{PratsPlanarDomains} to be used.

Section \ref{secQuasiconformal} gives the proof of Theorem \ref{theoInvertBeltrami}. In Subsection \ref{secOutline} one finds the outline of the proof via Fredholm Theory, reducing it to the compactness of a commutator  which is proven in Subsection \ref{secCompactness} and the compactness of $\chi_\Omega \Beurling \left(\chi_{\Omega^c} \Beurling^m\left(\chi_\Omega\cdot \right)\right)$ which is studied in Subsection \ref{secCompactnessDeath}.  Finally, Subsection \ref{secTechnical} is devoted to establishing a generalization of the results in \cite{MateuOrobitgVerdera} to be used in the last subsection.

\section{Preliminaries}\label{secPreliminaries}
\subsection{Some notation and well-known facts}\label{secNotation}
{\bf On inequalities:} 
When comparing two quantities $x_1$ and $x_2$ that depend on some parameters $p_1,\dots, p_j$ we will write 
$$x_1\leq C_{p_{i_1},\dots, p_{i_j}} x_2$$
if the constant $C_{p_{i_1},\dots, p_{i_j}} $ depends on ${p_{i_1},\dots, p_{i_j}}$. We will also write $x_1\lesssim_{p_{i_1},\dots, p_{i_j}} x_2$ for short, or simply $x_1\lesssim  x_2$ if the dependence is clear from the context or if the constants are universal. We may omit some of these variables for the sake of simplicity. The notation $x_1 \approx_{p_{i_1},\dots, p_{i_j}} x_2$ will mean that $x_1 \lesssim_{p_{i_1},\dots, p_{i_j}} x_2$ and $x_2 \lesssim_{p_{i_1},\dots, p_{i_j}} x_1$.

{\bf On polynomials:}
We write $\mathcal{P}^n(\R^d)$ for the vector space of real polynomials of degree smaller or equal than $n$ with $d$ real variables. If it is clear from the context we will just write $\mathcal{P}^n$.

{\bf On sets:}
Given two sets $A$ and $B$, we define their long distance as
$$\Dist(A,B):=\diam(A)+\diam(B)+\dist(A,B).$$
Given $x\in \R^d$ and $r>0$, we write $B(x,r)$ or $B_r(x)$ for the open ball centered at $x$ with radius $r$ and $Q(x,r)$ for the open cube centered at $x$ with sides parallel to the axis and side-length $2r$. Given any cube $Q$, we write $\ell(Q)$ for its side-length, and $rQ$ will stand for the cube with the same center but enlarged by a factor $r$. We will use the same notation for balls and one dimensional cubes, that is, intervals. 


 We call domain an open and connected subset of $\R^d$.
\begin{definition}\label{defLipschitz}
Given $n\geq 1$, we say that $\Omega\subset\C$ is a $(\delta,R)-C^{n-1,1}$ domain if given any $z\in\partial\Omega$, there exists a function $A_z\in C^{n-1,1}(\R)$ supported in $[-4R,4R]$ such that 
\begin{equation*}
\norm{A_z^{(j)}}_{L^\infty}\leq \frac{\delta}{R^{j-1}} \mbox{\,\,\,\, for every } 0\leq j \leq n,
\end{equation*}
and, possibly after a translation that sends $z$ to the origin and a rotation that brings the tangent at $z$ to the real line, we have that
$$\Omega\cap Q(0,R)= \{x+i\,y: y>A_z(x)\}.$$
In case $n=1$ the assumption of the tangent is removed (we say that $\Omega$ is a $(\delta,R)$-Lipschitz domain). 
\end{definition}

{\bf On measure theory:}
We denote the $d$-dimensional Lebesgue measure in $\R^d$ by $m$. At some point we use $m$ also to denote a natural number. We will write $dz$ for the form $dx+i\,dy$ and analogously $d\overline{z}=dx-i\,dy$, where $z=x+i\,y$. Thus, when integrating a function with respect to  the Lebesgue measure of a variable $z$ we will always use $dm(z)$ to avoid confusion,  or simply $dm$. 

{\bf On indices:}
In this text  $\N_0$ stands for the natural numbers including $0$. Otherwise we will write $\N$.  We will make wide use of the multiindex notation for exponents and derivatives. For $\alpha\in\Z^d$  its modulus is $|\alpha|=\sum_{i=1}^d|\alpha_i|$ and its factorial is $\alpha!=\prod_{i=1}^d \alpha_i!$. Given two multiindices $\alpha, \gamma\in \Z^d$ we write $\alpha\leq \gamma$ if $\alpha_i\leq \gamma_i$ for every $i$. We say $\alpha<\gamma$ if, in addition, $\alpha\neq\gamma$. Furthermore, we write
$${\alpha \choose \gamma}:=\prod_{i=1}^d {\alpha_i \choose \gamma_i}=\begin{cases}
\prod_{i=1}^d \frac{\alpha_i!}{\gamma_i!(\alpha_i-\gamma_i)!}  & \mbox{if }\alpha \in \N_0^d \mbox{ and }\vec{0}\leq \gamma \leq\alpha ,\\
0 & \mbox{otherwise.}
\end{cases}$$
For $x\in \R^d$ and $\alpha\in \Z^d$ we write $x^\alpha:= \prod x_i^{\alpha_i}$.  Given any $\phi\in C^\infty_c$ (infintitely many times differentiable with compact support in $\R^d$)  and $\alpha\in\N_0^d$ we write $D^\alpha\phi=\frac{\partial^{|\alpha|}}{\prod \partial_{x_i}^{\alpha_i}}\phi$. 

At some point we will use also use roman letter for multiindices, and then, to avoid confusion, we will use the vector notation $\vec{i},\vec{j}, \dots$

{\bf On complex notation}
For $z=x+i\,y \in \C$ we write $\real(z):=x$ and $\imag(z):=y$. Note that the symbol $i$ will be used also widely as a index for summations without risk of confusion. The multiindex notation will change slightly: for $z\in \C$ and $\alpha\in \Z^2$ we write $z^\alpha:=z^{\alpha_1}\overline{z}^{\alpha_2}$. 

We also adopt the traditional Wirtinger notation for derivatives, that is, given any $\phi\in C^\infty_c(\C)$, then 
$$\partial \phi (z):=\frac{\partial \phi}{\partial z}(z)=\frac12(\partial_x\phi-i\,\partial_y\phi) (z),$$
 and 
 $$\overline \partial \phi (z):=\frac{\partial\phi}{\partial \overline z}(z)=\frac12(\partial_x\phi+i\,\partial_y\phi) (z).$$
Thus, given any $\phi\in C^\infty_c(\C)$  and $\alpha\in\N_0^2$, we write $D^\alpha\phi=\partial^{\alpha_1}\overline\partial^{\alpha_2}\phi$.

{\bf On Sobolev spaces:}
For any open set $U$, every distribution $f\in \mathcal{D}'(U)$ and $\alpha\in\N_0^d$, the {\em distributional derivative} $D^\alpha_U f$ is the distribution defined by
$$\langle D^\alpha_U f,\phi\rangle:=(-1)^{|\alpha|}\langle f, D^\alpha \phi\rangle \mbox{\,\,\,\, for every }\phi \in C^\infty_c(U).$$
Abusing notation we will write $D^\alpha$ instead of $D^\alpha_U$ if it is clear from the context. If the distribution is regular, that is, if it coincides with an $L^1_{loc}$ function acting on $\mathcal{D}(U)$, then we say that $D^\alpha_U f$ is a {\em weak derivative} of $f$ in $U$. We write $|\nabla^n f|=\sum_{|\alpha|=n}|D^\alpha f|$.

Given numbers $n\in\N$, $1\leq p\leq\infty$ an open set $U\subset\R^d$ and an $L^1_{loc}(U)$ function $f$, we say that $f$ is in the Sobolev space $W^{n,p}(U)$ of smoothness $n$ and order of integrability $p$ if $f$ has weak derivatives $D^\alpha_U f\in L^p$ for every $\alpha\in\N_0^d$ with $|\alpha|\leq n$. When $\Omega$ is a Lipschitz domain, we will use the norm
$$\norm{f}_{W^{n,p}(\Omega)}=\norm{f}_{L^p(\Omega)}+\norm{\nabla^n f}_{L^p(\Omega)},$$
which has other equivalent expressions such as
\begin{equation}\label{eqEquivalenceNormsSobolev}
\norm{f}_{W^{n,p}(\Omega)}\approx \norm{f}_{L^p(\Omega)}+\sum_{|\alpha|\leq n}\norm{D^\alpha f}_{L^p(\Omega)}\approx \norm{f}_{L^p(\Omega)}+\sum_{j=1}^{d}\norm{\partial_j^n f}_{L^p(\Omega)}
\end{equation}
(see \cite[Theorem 4.2.4]{TriebelInterpolation}) or, if $\Omega$ is an extension domain,
$$\norm{f}_{W^{n,p}(\Omega)}\approx \inf_{F: F|_\Omega\equiv f}\norm{F}_{W^{n,p}(\R^d)}.$$
From \cite{Jones}, we know that uniform domains (and in particular, Lipschitz domains) are Sobolev extension domains for any indices $n\in \N$ and $1\leq p \leq\infty$.

%
%
{\bf On Whitney coverings:}
Given a domain $\Omega$, we say that a collection of open dyadic cubes $\mathcal{W}$ is a {\rm Whitney covering} of $\Omega$ if they are disjoint, the union of the cubes and their boundaries is $\Omega$, there exists a constant $C_{\mathcal{W}}$ such that 
$$C_\mathcal{W} \ell(Q)\leq \dist(Q, \partial\Omega)\leq 4C_\mathcal{W}\ell(Q),$$
two neighbor cubes $Q$ and $R$ (i.e., $\overline Q\cap \overline R\neq\emptyset$) satisfy $\ell(Q)\leq 2 \ell(R)$, and the family $\{20 Q\}_{Q\in\mathcal{W}}$ has finite overlapping. The existence of such a covering is granted for any open set different from $\R^d$ and in particular for any domain as long as $C_\mathcal{W}$ is big enough (see \cite[Chapter 1]{SteinPetit} for instance).

{\bf On the Leibniz rule:}
Given a domain $\Omega\subset \R^d$, a function $f\in W^{n,p}(\Omega)$ and a multiindex $\alpha\in \N_0^d$ with $|\alpha|\leq n$, if $\phi\in C^\infty_c(\Omega)$, then $\phi\cdot f\in  W^{n,p}(\Omega)$ and 
\begin{equation}\label{eqLeibniz}
D^\alpha (\phi\cdot f)=\sum_{\gamma\leq \alpha}{\alpha\choose\gamma} D^\gamma\phi D^{\alpha-\gamma} f
\end{equation}
 (see \cite[Section 5.2.3]{Evans}).
 
{\bf On Green's formula:}
The Green Theorem can be written in terms of complex derivatives (see \cite[Theorem 2.9.1]{AstalaIwaniecMartin}). Let $\Omega$ be a bounded Lipschitz domain. If $f, g\in W^{1,1}(\Omega)\cap C(\overline{\Omega})$, then 
\begin{equation}\label{eqGreen}
\int_\Omega \left(\partial f + \overline \partial g\right) \,dm=\frac{i}2\left(\int_{\partial\Omega} f(z) \, d\overline{z}- \int_{\partial\Omega}g(z)\, dz\right).
\end{equation}

{\bf On the Sobolev Embedding Theorem:}
We state a reduced version of the Sobolev Embedding Theorem for Lipschitz domains (see \cite[Theorem 4.12, Part II]{AdamsFournier}). 
For each Lipschitz domain $\Omega\subset \R^d$ and every $p>d$, there is a continuous embedding of the Sobolev space $W^{1,p}(\Omega)$ into the H\"older space $C^{0,1-\frac{d}{p}}(\overline\Omega)$. That is, writing
$$\norm{f}_{C^{0,s}(\overline\Omega)}=\norm{f}_{L^\infty(\Omega)}+\sup_{\substack{x,y\in\overline\Omega\\x\neq y}}\frac{|f(x)-f(y)|}{|x-y|^s}\mbox{\,\,\,\, for $0<s\leq 1$},$$
we have that for every $f\in W^{1,p}(\Omega)$, 
\begin{equation}\label{eqSobolevEmbedding}
\norm{f}_{L^\infty(\Omega)}\leq \norm{f}_{C^{0,1-\frac dp}(\overline\Omega)} \leq C_\Omega \norm{f}_{W^{1,p}(\Omega)}.
\end{equation}

{\bf On inequalities:}
We will use Young's inequality. It states that for measurable functions $f$ and $g$, we have that
\begin{equation}\label{eqYoung}
\norm{f * g}_{L^q}\leq\norm{f}_{L^r}\norm{g}_{L^p}
\end{equation}
for $1\leq p,q,r\leq \infty$ with $\frac{1}{q}=\frac{1}{p}+\frac{1}{r}-1$ (see \cite[Appendix A2]{SteinPetit}).

\subsection{On chains and approximating polynomials}\label{secChainsAndPolynomials}
In the proof of Lemmas \ref{lemCompactness} and \ref{lemCompactnessDeath} we will use some techniques from \cite[Sections 3 and 4]{PratsTolsa}. We sum up some results here and refer the reader to that paper for the details. First we need the concept of `chain of cubes', which can be seen as some kind of hyperbolic path between the centers of those cubes. Along this section and the following one, we consider $\R^d$ as the ambient space, the necessary modifications to Definition \ref{defLipschitz} are left to the reader.

\begin{figure}[ht]
 \centering
 {\includegraphics[width=0.9\textwidth]{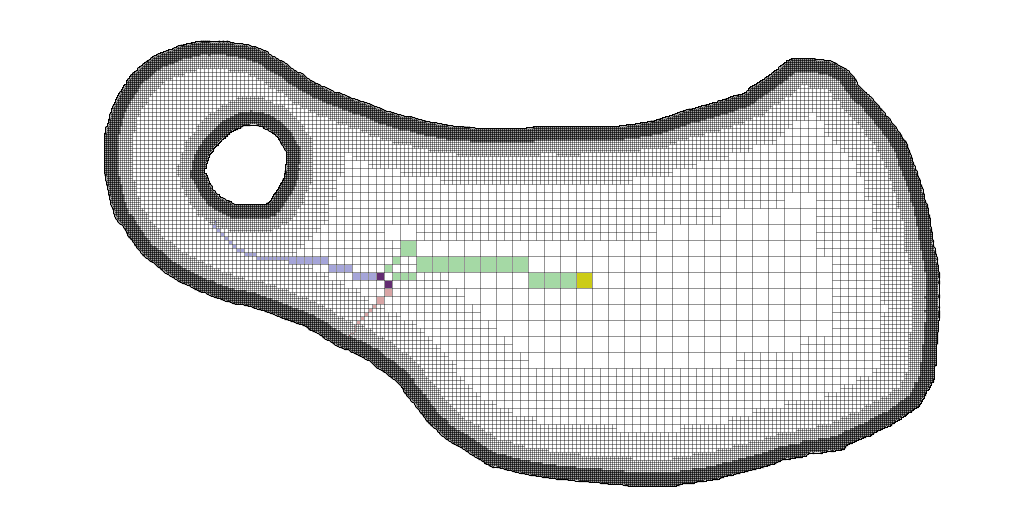}}    
  \caption{A Whitney decomposition of a Lipschitz domain with an admissible chain. In green, the prolongations to $Q_0$ (see Remark \ref{remChain}).}\label{figCovering}
\end{figure}

\begin{remark}[see \cite{PratsTolsa}]\label{remChain}
Consider a Lipschitz domain $\Omega \subset \R^d$, a Whitney covering $\mathcal{W}$, and a fixed Whitney cube $Q_0\in \mathcal{W}$ with size comparable to the diameter of $\Omega$. For every pair of Whitney cubes $Q$ and $S$ there exists an {\em admissible chain} $[Q,S]\in \bigcup_{M=1}^\infty \mathcal{W}^M$, that is,  satisfying the following properties:
\begin{enumerate}
\item The chain $[Q,S]=(Q_1,\dots,Q_M)$ satisfies that $Q_1=Q$, $Q_M=S$ and for any $1\leq j<M$, the cubes $Q_j$ and its {\rm next cube} in the chain $[Q,S]$, $\mathcal{N}(Q_j):=Q_{j+1}$ are neighbors. Abusing the notation, we also write $[Q,S]$ for the set $\{Q_1,\dots,Q_M\}$.

\item The length of the chain $\ell([Q,S]):=\sum_{j=1}^M \ell(Q_j)$ satisfies that $\ell([Q,S])\approx \Dist(Q,S)$, with constants depending only on $d$ and the Lipschitz character of $\Omega$.

\item If $M>1$, there exist two neighbor cubes $Q_S,S_Q\in [Q,S]$ such that the subchains $[Q,Q_S]$ and $[S_Q,S]$ are disjoint, the union $[Q,Q_S]\cup[S_Q,S]=[Q,S]$ and there are two admissible chains $[Q,Q_0]$ and $[Q_0,S]$ such that the subchains $[Q,Q_S]\subset [Q,Q_0]$ and $[S_Q,S]\subset [Q_0,S]$. In other words, $[Q,Q_S]$ is the ``ascending'' subchain and $[S_Q,S]$ is the ``descending'' subchain (see Figure \ref{figCovering}). 

\item  For $P \in [Q,Q_S]$, $L\in[S_Q,S]$ we have that
\begin{equation}\label{eqDPQFar}
\Dist(P,S)\approx \Dist(Q,S) \approx \Dist(Q,L).
\end{equation}
Moreover
\begin{equation}\label{eqDPQClose}
\Dist(P,Q)\approx \ell(P) \quad\quad\mbox{ and }\quad\quad \Dist(L,S)\approx \ell(L).
\end{equation}
In particular, 
\begin{equation*}
\ell(Q_S)\approx\ell(S_Q)\approx \Dist(Q,S)\approx \Dist(Q,Q_S)\approx \Dist(Q_S,S).
\end{equation*}
All the constants depend only on $d$ and the Lipschitz character of $\Omega$.
\end{enumerate}
\end{remark}

\begin{definition}\label{defShadow}
If $Q, S\in [P,Q_0]$ for some Whitney cube $P$ and $\mathcal{N}^j(Q)=S$ for a certain $j \in \N_0$, then we say that $Q\leq S$. 

We define the \emph{shadow} of $Q$ as $\SH_\rho(Q):=\{S: \Dist(S,Q)\leq \rho Q\}$, and its ``realization'' is the region $\Sh_\rho(Q):=\bigcup_{S\in\SH_\rho(Q)} S$. For $\rho_0$ big enough, we have that every Whitney cube $Q$ satisfies that
 $$\{S: S\leq Q\} \subset \SH_{\rho_0}(Q).$$
 We will then write ${\Sh} (Q):={\Sh}_{\rho_0}(Q)$ (see Figure \ref{figShadow}).
\end{definition}

\begin{figure}[ht]
 \centering
 {\includegraphics[width=0.9\textwidth]{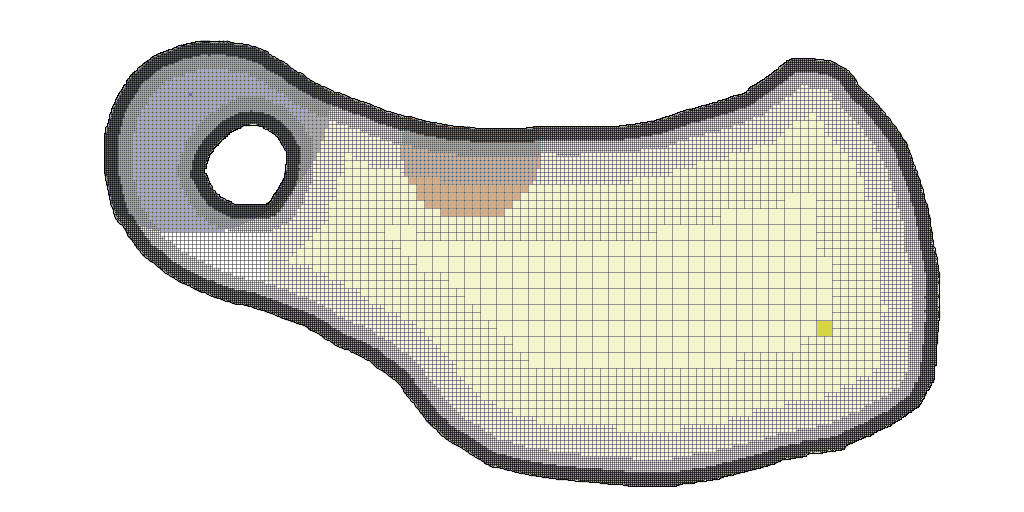}}    
  \caption{A Whitney decomposition of a Lipschitz domain with the shadows of three different cubes (see Definition \ref{defShadow}).}\label{figShadow}
\end{figure}

Let us recall the definition of the non-centered Hardy-Littlewood maximal operator. Given $f\in L^1_{loc}(\R^d)$ and $x\in\R^d$, we define $Mf(x)$ as the supremum of the mean of $f$ in cubes containing $x$, that is,
$$Mf(x)=\sup_{Q\ni x} \frac{1}{|Q|} \int_Q f(y) \, dy.$$
It is a well known fact that this operator is bounded on $L^p$ for $1<p<\infty$. We are interested also in the properties of the maximal function exposed in \cite{PratsTolsa}.
\begin{lemma}\label{lemmaximal}
Assume that $g\in L^1_{loc}(\R^d)$ and $r>0$. For every $Q\in\mathcal{W}$, we have
\begin{enumerate}[1)]
\item If $\eta>0$, 
\begin{equation}\label{eqMaximalFar}
 \sum_{S:\Dist(Q,S)>r}  \frac{\int_S g(x) \, dx}{D(Q,S)^{d+\eta}}\lesssim \frac{\inf_{y\in Q} Mg(y)}{r ^\eta}.
 \end{equation}
\item If $\eta>0$,
\begin{equation}\label{eqMaximalClose} 
 \sum_{S:\Dist(Q,S)<r}  \frac{\int_S g(x) \, dx}{D(Q,S)^{d-\eta}}\lesssim \inf_{y\in Q} Mg(y) \,r^\eta.
 \end{equation}
\item In particular, 
\begin{equation*}
 \sum_{S: S<Q} \int_S g(x) \, dx\lesssim \inf_{y\in Q} Mg(y) \, \ell(Q)^d.
\end{equation*}
\end{enumerate}
\end{lemma}

We will also use some approximating polynomials of a Sobolev function $f$ around $3Q$. Namely, given a function $f\in W^{n,p}(Q)$, we define $\mathbf{P}_{Q}^{n} f$ as the unique polynomial such that for every multiindex $\alpha$ with $|\alpha|\leq n$, we have that
$$\int_Q D^\alpha f\, dm=\int_Q D^\alpha \mathbf{P}_Q^n f\, dm.$$
These polynomials have the following properties:
\begin{enumerate}
\item Let $z_Q$ be the center of $Q$. If we consider the Taylor expansion of $\mathbf{P}_{3Q}^{n-1} f$ at $z_Q$, 
\begin{equation}\label{eqTaylorExpansion}
\mathbf{P}_{3Q}^{n-1} f (z)= \sum_{|\gamma|<n}m_{Q,\gamma} (z-z_Q)^\gamma,
\end{equation}
then the coefficients $m_{Q,\gamma}$ are bounded by
\begin{equation}\label{eqMGammaBounded}
|m_{Q,\gamma}| \lesssim_{d,n} \norm{f}_{W^{n-1,\infty}(3Q)} (1+\ell(Q)^{n-1}). 
\end{equation}

\item Let us assume that, in addition, the function $f$ is in the Sobolev space $W^{n,p}(3Q)$ for a certain $1\leq p \leq\infty$. Given $0\leq j\leq n$, if we have a smooth function $\varphi \in C^\infty(3Q)$ satisfying $\norm{\nabla^i\varphi}_{L^\infty(3Q)}\lesssim\frac{1}{\ell(Q)^j}$ for  $0\leq i\leq j$, then we have the Poincar\'e inequality
 \begin{equation}\label{eqPoincare}
 \norm{\nabla^j\left(\left(f-\mathbf{P}_{3Q}^{n-1} f\right) \varphi \right)}_{L^p(3Q)}\leq C \ell(Q)^{n-j} \norm{\nabla^n f}_{L^p(3Q)} .
 \end{equation}

\item Given a domain with a Whitney covering $\mathcal{W}$, two Whitney cubes $Q, S\in \mathcal{W}$, an admissible chain $[S,Q]$ as in Remark \ref{remChain},  and $f\in W^{n,p}(\Omega)$, we have that
\begin{equation}\label{eqChain}
\norm{f-\mathbf{P}^{n-1}_{3Q} f}_{L^1(S)}  \lesssim_{d,n} \sum_{P\in [S,Q]}\frac{\ell(S)^d D(P,S)^{n-1}}{\ell(P)^{d-1}}\norm{\nabla^n f}_{L^1(5P)}.
\end{equation}
\end{enumerate}

\begin{lemma}\label{lemTwoFoldedFunctionGuay}
Consider a Lipschitz domain $\Omega\subset \R^d$ with Whitney covering $\mathcal{W}$, two functions $f\in L^p(\Omega)$ and $g\in L^{p'}(\Omega)$ and $\rho\geq 1$. Then
$$A_\rho(f,g):=\sum_{Q,S\in\mathcal{W}}\sum_{P\in [S,Q]}\frac{\ell(S)^d D(P,S)^{\rho-1}\norm{f}_{L^1(20P)}\norm{g}_{L^1(20Q)}}{\ell(P)^{d-1} D(Q,S)^{\rho+d}}\lesssim \norm{f}_{L^p(\Omega)}\norm{g}_{L^{p'}(\Omega)}.$$
\end{lemma}
\begin{proof}
Using that $P\in [S,Q]$ implies $\Dist(P,S)\lesssim \Dist(Q,S)$ (see Remark \ref{remChain}), we get
\begin{align*}
A_\rho(f,g)
	&\lesssim \left(\sum_{Q,S\in\mathcal{W}} \sum_{P\in [S,S_Q]}+ \sum_{Q,S\in\mathcal{W}}\sum_{P\in [Q_S,Q]}  \right) \frac{\ell(S)^d \norm{f}_{L^1(20P)}\norm{g}_{L^1(20Q)}}{\ell(P)^{d-1} D(Q,S)^{d+1}} \\
	& = A^{(1)}(f,g)+A^{(2)}(f,g).
\end{align*}

We consider first the term $A^{(1)}(f,g)$ where the sum is taken with respect to  cubes $P\in [S,S_Q]$ and, thus, by \rf{eqDPQFar} the long distance $\Dist(Q,S)\approx \Dist(P,Q)$. Moreover, we have $S\in\SH(P)$ by Definition \ref{defShadow}. Thus, rearranging the sum, 
\begin{align*}
A^{(1)}(f,g)
	&\lesssim 
 \sum_{P\in\mathcal{W}}  \frac{\norm{f}_{L^1(20P)}}{\ell(P)^{d-1}} \sum_{Q\in\mathcal{W}}  \frac{\norm{g}_{L^1(20Q)}}{D(Q,P)^{d+1}}  \sum_{S \in \SH(P)} \ell(S)^d.
\end{align*}
By  Definition \ref{defShadow} again
$$\sum_{S \in \SH(P)} \ell(S)^d \approx \ell(P)^d$$
and, by \rf{eqMaximalFar} and the finite overlapping of the cubes $\{20 Q\}_{Q\in\mathcal{W}}$, we get
$$ \sum_{Q\in\mathcal{W}}  \frac{\norm{g}_{L^1(20Q)}}{D(Q,P)^{d+1}}\lesssim \frac{\inf_{x\in20P} Mg(x)}{\ell(P)}.$$

Next we perform a similar argument with $A^{(2)}(f,g)$. Note that when $P\in[Q_S,Q]$, we have $\Dist(Q,S) \approx \Dist(P,S)$ and $Q \in \SH(P)$, leading to
\begin{align*}
A^{(2)}(f,g)
	&\lesssim 
 \sum_{P\in\mathcal{W}}  \frac{\norm{f}_{L^1(20P)}}{\ell(P)^{d-1}} \sum_{Q \in \SH(P)}  \norm{g}_{L^1(20Q)}  \sum_{S\in\mathcal{W}} \frac{\ell(S)^d}{D(P,S)^{d+1}}.
\end{align*}
By \rf{eqMaximalClose} we get
$$\sum_{Q \in \SH(P)}  \norm{g}_{L^1(20Q)} \lesssim  \inf_{x\in20P} Mg(x)\, \ell(P)^d$$
and, applying \rf{eqMaximalFar} to the characteristic function of the domain,
$$ \sum_{S\in\mathcal{W}}  \frac{\ell(S)^d}{D(P,S)^{d+1}}\approx \frac{1}{\ell(P)}.$$
Thus, 
\begin{align*}
A_\rho(f,g)
	&\lesssim
 \sum_{P\in\mathcal{W}}  \frac{\norm{f}_{L^1(20P)}}{\ell(P)^{d-1}} \frac{\inf_{20P} Mg}{\ell(P)} \ell(P)^d \lesssim  \sum_{P\in\mathcal{W}} \norm{f \cdot Mg}_{L^1(20P)}.
\end{align*}
By H\"older inequality and the boundedness of the Hardy-Littlewood maximal operator in $L^{p'}$, 
\begin{align*}
A_\rho(f,g)
	&\lesssim \left(\sum_{P\in\mathcal{W}} \norm{ f}_{L^p(20P)}^p\right)^{1/p} \left(
 \sum_P \norm{Mg}_{L^{p'}(20P)}^{p'}\right)^{1/p'} \lesssim \norm{ f}_{L^p(\Omega)}\norm{g}_{L^{p'}(\Omega)} .
\end{align*}
\end{proof}

\subsection{Function spaces}\label{secSpaces}
Next we recall some definitions and results on the function spaces that we will use. For a complete treatment we refer the reader to \cite{TriebelTheory} and \cite{RunstSickel}.

\begin{definition}\label{defCollection}Let $\{\psi_j\}_{j=0}^\infty\subset C^\infty_c(\R^d)$ be a family of radial functions such that
\begin{equation*}
\left\{ 
\begin{array}{ll}
\supp \,\psi_0 \subset \DDD(0,2), & \\
\supp \,\psi_j \subset \DDD(0,2^{j+1})\setminus \DDD(0,2^{j-1}) & \mbox{ if $j\geq 1$},\\	
\end{array}
\right.
\end{equation*}
for all multiindex $\alpha\in \N^d$ there exists a constant $c_\alpha$ such that
\begin{equation*}
\norm{D^\alpha \psi_j}_\infty \leq \frac{c_\alpha}{2^{j |\alpha|} } \mbox{\,\,\, for every $j\geq 0$}
\end{equation*}
and
\begin{equation*}
\sum_{j=0}^\infty \psi_j(x)=1 \mbox{\,\,\, for every $x\in\R^d$.}
\end{equation*}
\end{definition}
 
 \begin{definition}
Given any Schwartz function $\psi \in \mathcal{S}(\R^d)$ its Fourier transform is
$$F\psi(\zeta)=\int_{\R^d} e^{-2\pi i x\cdot \zeta} \psi(x) dm(x).$$
This notion extends to the tempered distributions $\mathcal{S}(\R^d)'$ by duality.
 
Let $s \in \R$,  $1\leq p\leq \infty$, $1\leq q\leq\infty$. Then we define the non-homogeneous Besov space $B^s_{p,q}(\R^d)$ as the set of tempered distributions $f\in \mathcal{S}'(\R^d)$ such that 
\begin{equation*}
\norm{f}_{B^s_{p,q}}=\norm{\left\{2^{sj}\norm{F^{-1}\psi_j F f}_{L^p}\right\}}_{l^q}<\infty.
\end{equation*}
\end{definition}
These norms are equivalent for different choices of $\{\psi_j\}$.

Consider the boundary of a Lipschitz domain $\Omega \subset \C$. When it comes to the Besov space $B^s_{p,q}(\partial \Omega)$ we can just define it using the arc parameter of the curve, $z:I \to \partial\Omega$ with $|z'(t)|=1$ for all $t$.  We also use an auxiliary bump function $\varphi_\Omega:\R\to\R$ such that $\varphi_\Omega|_{2 I}\equiv 1$ and $\varphi_\Omega|_{(4 I)^c}\equiv 0$. 
Then, if $1\leq p,q< \infty$,  we define naturally the homogeneous Besov norm on the boundary of $\Omega$ as
\begin{equation*}
\norm{f}_{B^s_{p,q}(\partial\Omega)} :=\norm{(f\circ z) \varphi_\Omega}_{B^s_{p,q}(\R)}.
\end{equation*}
Note that if the domain is bounded, then $I$ is a finite interval with length equal to the length of the boundary of $\Omega$ and we need to extend $z$ periodically to $\R$ in order to have a sensible definition above.
For more information on these norms, we refer the reader to \cite[Section 2.3]{PratsPlanarDomains}.

\begin{theorem}\label{theoRunstSickel}
Let $n\in \N$ and $1<p<\infty$ with $np>d$. If $\Omega \subset \R^d$ is an extension domain, then for every pair $f,g\in W^{n,p}(\Omega)$ we have that
$$\norm{f\, g}_{W^{n,p}(\Omega)}\leq C_{d,n,p,\Omega} \norm{f}_{W^{n,p}(\Omega)}\norm{ g}_{W^{n,p}(\Omega)}.$$
Moreover, if  $\Omega$ is a Lipschitz domain and  $p>d$, then for $m\geq n$ we have that
$$\norm{f^m}_{W^{n,p}(\Omega)}\leq C_{d,n,p,\Omega} \, m^n \, \norm{f}_{L^\infty(\Omega)}^{m-n}  \norm{f}_{W^{n,p}(\Omega)}^n .$$
\end{theorem}
\begin{proof}
We have that $W^{n,p}(\R^d)$ is a multiplicative algebra (see \cite[Section 4.6.4]{RunstSickel}), that is, if $f,g\in W^{n,p}(\R^d)$, then
$$\norm{f\, g}_{W^{n,p}}\leq C_{d,n,p} \norm{f}_{W^{n,p}}\norm{ g}_{W^{n,p}}.$$
Since $\Omega$ is an extension domain, we have a bounded operator $E:W^{n,p}(\Omega)\to W^{n,p}(\R^d)$ such that $(Ef)|_\Omega=f|_\Omega$ for every $f\in W^{n,p}(\Omega)$. The first property is a consequence of this fact.

To prove the second property, first assume that $f\in C^\infty(\overline{\Omega})$. By \rf{eqEquivalenceNormsSobolev} we only need to prove that $\norm{\partial_k^n(f^m)}_{L^p(\Omega)} \leq C_{d,n,p,\Omega} m^n \left( \norm{f}_{L^\infty(\Omega)}^{m-n}  \norm{f}_{W^{n,p}(\Omega)}^n \right)$ for $1\leq k\leq d$. Without loss of generality, we will assume $k=1$. By the Leibniz' rule, it is an exercise to check that
\begin{equation}\label{eqBreakEveryCombination}
\partial_1^n(f^m)= f^{m-n}\sum_{\substack{\vec{j}\in \N_0^n\\j_i\geq j_{i+1} \text{ for } 1\leq i <n\\ |\vec{j}|=n}} c_{\vec{j},m}\prod_{i=1}^{n} \partial_1^{j_i} f,
\end{equation}
with $c_{\vec{j},m}>0$ and $\sum_{\vec{j}} c_{\vec{j},m} = m^n$. 
Consider $\vec{j}=(n,0,\cdots,0)$. Then, by  \rf{eqSobolevEmbedding}, that is, the Sobolev embedding Theorem, we get
\begin{equation}\label{eqSimplestCase}
 \norm{\prod_{i=1}^{n} \partial_1^{j_i} f}_{L^p(\Omega)}=\norm{ \partial_1^{n} f \, f^{n-1}}_{L^p(\Omega)}\leq \norm{\partial_1^{n} f}_{L^p(\Omega)}\norm{f}_{L^\infty(\Omega)}^{n-1} \lesssim_{\Omega,p} \norm{f}_{W^{n,p}(\Omega)}^{n} .
 \end{equation}
For $\vec{j}\neq(n,0,\cdots,0)$, the indices $j_i<n$ for $1\leq i \leq n$ and we use \rf{eqSobolevEmbedding} again to state that
\begin{equation}\label{eqOtherCases}
 \norm{\prod_{i=1}^{n} \partial_1^{j_i} f}_{L^p(\Omega)} \leq \prod_{i=1}^{n} \norm{\partial_1^{j_i} f}_{L^{\infty}(\Omega)}|\Omega|^\frac{1}{p} \lesssim_{\Omega,n,p}  \prod_{i=1}^{n} \norm{\partial_1^{j_i} f}_{W^{1,p}(\Omega)} \leq \norm{f}_{W^{n,p}(\Omega)}^{n} .
 \end{equation}
By \rf{eqBreakEveryCombination}, \rf{eqSimplestCase}, \rf{eqOtherCases} and the triangle inequality, we get that
$$\norm{\partial_1^n(f^m)}_{L^p(\Omega)}\leq \norm{f^{m-n}}_{L^\infty(\Omega)}\sum_{\substack{\vec{j}\in \N_0^n\\j_i\geq j_{i+1} \mbox{ for } 1\leq i <n\\ |\vec{j}|=n}} c_{\vec{j},m}\norm{\prod_{i=1}^{n} \partial_1^{j_i} f}_{L^p(\Omega)}\lesssim  m^n\norm{f}_{L^\infty(\Omega)}^{m-n} \norm{f}_{W^{n,p}(\Omega)}^{n} .$$
By an approximation procedure this property applies to every $f\in W^{n,p}(\Omega)$.
\end{proof}

\subsection{A family of convolution operators in the plane}\label{secOperators}
\begin{definition}\label{defOperador}
Consider a function $K \in L^1_{loc}(\C\setminus\{0\})$. For any $f\in L^1_{loc}$ we define
$$T^K f(z)=\lim_{\varepsilon\to 0}\int_{\C \setminus B_\varepsilon(z)}K(z-w)f(w) \,dm(w)$$
as long as the limit exists, for instance, when $K$ is bounded away from $0$, $f\in L^1$ and $z\notin \supp(f)$ or when $f=\chi_U$ for an open set $U$ with $z\in U$, $\int_{B_\varepsilon(0)\setminus B_{\varepsilon'}(0)} K \, dm =0$ for every $\epsilon>\varepsilon'>0$ and $K$ is integrable at infinity. We say that $K$ is the kernel of $T^K$. 

For any multiindex $\gamma \in \Z^2$, we will consider $K^\gamma(z)=z^{\gamma}=z^{\gamma_1}\overline{z}^{\gamma_2}$ and then we abbreviate $T^\gamma f:=T^{K^\gamma} f$, that is,
\begin{equation*}
T^\gamma f(z)=\lim_{\varepsilon\to 0}\int_{\C \setminus B_\varepsilon(z)}(z-w)^\gamma f(w) \,dm(w)
\end{equation*}
as long as the limit exists.

For any operator $T$ and any domain $\Omega$, we can consider $T_\Omega f= \chi_\Omega \, T(\chi_\Omega\, f)$.
\end{definition}

\begin{example}\label{exBeurlingIterate}
As the reader may have observed, the Beurling and the Cauchy transforms are in the above family of operators. Namely, when $K(z)=z^{-2}$, that is, for $\gamma=(-2,0)$, then $\frac{-1}{\pi}T^\gamma$ is the Beurling transform. The operator $\frac1\pi T^{(-1,0)}$ coincides with the Cauchy transform.

Consider the iterates of the Beurling transform $\Beurling^m$ for $m>0$. For every $f \in L^p$ and $z\in \C$ we have
\begin{align*}
\Beurling^mf(z)
	& =\frac{(-1)^m m}{\pi}\lim_{\varepsilon\to 0}\int_{|z-\tau|>\varepsilon}\frac{(\overline{z-\tau})^{m-1}}{(z-\tau)^{m+1}}f(\tau)\, dm(\tau) =\frac{(-1)^m m}{\pi}T^{(-m-1,m-1)}f(z).
\end{align*}
That is, for $\gamma=(\gamma_1, \gamma_2)$ with $\gamma_1+\gamma_2=-2$ and $\gamma_1\leq -2$, the operator $T^\gamma$ is an iteration of the Beurling transform modulo constant (see \cite[Section 4.2]{AstalaIwaniecMartin}), and it maps $L^p(U)$ to itself for every open set $U$. If $\gamma_2\leq -2$, then $T^\gamma$ is an iterate of the conjugate Beurling transform and it is bounded in $L^p$ as well.
 \end{example}

Let us sum up some properties of the Cauchy transform which will be useful in the subsequent sections (see  \cite[Theorems 4.3.10, 4.3.12, 4.3.14]{AstalaIwaniecMartin}). We write $I_\Omega g:= \chi_\Omega\, g$ for every $g\in L^1_{loc}$.
\begin{theorem}\label{theoCauchy}
Let $1<p<\infty$. Then
\begin{itemize}
\item For every $f\in L^p$, we have that $\partial \Cauchy f= \Beurling f$ and $\overline{\partial}\Cauchy f= f$. 
\item For every function $f\in L^1$ with compact support, we have that if $p>2$ then
\begin{equation}\label{eqCauchyCompact}
\norm{\Cauchy f}_{L^p}\lesssim_p \diam(\supp(f))\norm{f}_{L^p}.
\end{equation}
\item Let $\Omega$ be a bounded open subset of $\C$. Then,  we have that
\begin{equation}\label{eqCauchyLocalized}
I_\Omega \circ \Cauchy : L^p(\C)\to W^{1,p}(\Omega)
\end{equation}
is bounded.
\end{itemize}
\end{theorem}

In the companion article \cite{PratsPlanarDomains}, we proved the following theorem.
\begin{theorem}[see {\cite[Theorem 3.16]{PratsPlanarDomains}}]\label{theoGeometricPGtr2}
Consider $p>2$,   $n\geq 1$ and let $\Omega$ be a Lipschitz domain with parameterizations in $B^{n+1-1/p}_{p,p}$. Then, for every $\epsilon>0$ there exists a constant $C_\epsilon$ depending on $p$, $n$, $\Omega$ and $\epsilon$ such that for every multiindex $\gamma\in \Z^2\setminus\{(-1,-1)\}$ with $\gamma_1+\gamma_2\geq -2$, one has 
\begin{equation*}
\norm{T^\gamma_\Omega}_{W^{n,p}(\Omega)\to W^{n+{\gamma_1+\gamma_2+2},p}(\Omega)}\leq C_\epsilon |\gamma|^{n+{\gamma_1+\gamma_2+2}}\left( \norm{N}_{B^{n-1/p}_{p,p}(\partial\Omega)}+(1+\epsilon)^{|\gamma|}\right) + \diam(\Omega)^{\gamma_1+\gamma_2+2}.
\end{equation*}

In particular (see Example \ref{exBeurlingIterate}), for $m\in\N$ we have that  $(\Beurling^m)_\Omega$ is bounded in $W^{n,p}(\Omega)$, with norm
\begin{equation*}
\norm{(\Beurling^m)_\Omega}_{W^{n,p}(\Omega)\to W^{n,p}(\Omega)}\leq C_\epsilon m^{n+1}\left( \norm{N}_{B^{n-1/p}_{p,p}(\partial\Omega)}+ (1+\epsilon)^{m}\right).
\end{equation*}
\end{theorem}

\section{Quasiconformal mappings}\label{secQuasiconformal}
\subsection{Proof of Theorem \ref{theoInvertBeltrami}}\label{secOutline}
Consider $m\in \N$. Recall that $(\Beurling^m)_\Omega g =\chi_\Omega \Beurling^m(\chi_\Omega g)$ for $g \in L^1_{loc}$ (see Definition \ref{defOperador}) and $I_\Omega g=\chi_\Omega\, g$.  Note that $I_\Omega$ is the identity in $W^{n,p}(\Omega)$.  
Let us define 
$P_m:=I_\Omega+\mu \Beurling_\Omega+(\mu \Beurling_\Omega)^2+\cdots + (\mu \Beurling_\Omega)^{m-1}$. Since $W^{n,p}(\Omega)$ is a multiplicative algebra (by Theorem \ref{theoRunstSickel}), we have that $P_m$ is bounded in $W^{n,p}(\Omega)$. Note that
\begin{equation}\label{eqPmToIMu}
P_m\circ (I_\Omega-\mu \Beurling_\Omega)=(I_\Omega-\mu \Beurling_\Omega)\circ P_m=I_\Omega-(\mu \Beurling_\Omega)^m,
\end{equation}
and
\begin{align}\label{eqFactorizePm}
I_\Omega-(\mu \Beurling_\Omega)^m
\nonumber	&=(I_\Omega-\mu^m(\Beurling^m)_\Omega)+\mu^m((\Beurling^m)_\Omega-(\Beurling_\Omega)^m)+(\mu^m (\Beurling_\Omega)^m-(\mu \Beurling_\Omega)^m)\\
			& =A^{(1)}_m+\mu^m A^{(2)}_m+ A^{(3)}_m.
\end{align}
Note the difference between $(\Beurling_\Omega)^m g= \chi_\Omega \Beurling(\dots\chi_\Omega \Beurling(\chi_\Omega \Beurling(\chi_\Omega g)))$ and $(\Beurling^m)_\Omega g =\chi_\Omega \Beurling^m(\chi_\Omega g)$.  Next we will see that for $m$ large enough, the operator $I_\Omega-(\mu \Beurling_\Omega)^m$ is the sum of an invertible operator and a compact one.

First we will study the compactness of $A^{(3)}_m=\mu^m (\Beurling_\Omega)^m-(\mu \Beurling_\Omega)^m$. To start, writing $[\mu, \Beurling_\Omega](\cdot)$ for the commutator $\mu \Beurling_\Omega(\cdot)-\Beurling_\Omega(\mu \cdot)$ we have the telescopic sum
\begin{align*}
A^{(3)}_m	
	& =\sum_{j=1}^{m-1}\mu^{m-j}[\mu, \Beurling_\Omega]\left(\mu^{j-1}(\Beurling_\Omega)^{m-1}\right)+(\mu \Beurling_\Omega) (\mu^{m-1} (\Beurling_\Omega)^{m-1}-(\mu \Beurling_\Omega)^{m-1})\\
			& =\sum_{j=1}^{m-1}\mu^{m-j}[\mu, \Beurling_\Omega]\left(\mu^{j-1}(\Beurling_\Omega)^{m-1}\right)+(\mu \Beurling_\Omega) A^{(3)}_{m-1}.
\end{align*}
Arguing by induction we can see that $A^{(3)}_m $ can be expressed as a sum of operators bounded in $W^{n,p}(\Omega)$ which have $[\mu, \Beurling_\Omega]$ as a factor. It is well-known that the compactness of a factor implies the compactness of the operator (see for instance \cite[Section 4.3]{Schechter}). Thus, the following lemma, which we prove in Section \ref{secCompactness} implies the compactness of $A^{(3)}_m$.
\begin{lemma}\label{lemCompactness}
The commutator $[\mu, \Beurling_\Omega]$ is compact in $W^{n,p}(\Omega)$.
\end{lemma}

Consider now $A^{(2)}_m=(\Beurling^m)_\Omega-(\Beurling_\Omega)^m$. We define the operator $\mathcal{R}_m g:=\chi_\Omega \Beurling\left(\chi_{\Omega^c} \Beurling^{m-1}(\chi_\Omega \, g)\right)$ whenever it makes sense. This operator can be understood as a (regularizing) reflection with respect to the boundary of $\Omega$. For every $g\in W^{n,p}(\Omega)$ we have that
\begin{align*}
A^{(2)}_m g
			& =\chi_\Omega\left( \Beurling\left((\chi_\Omega + \chi_{\Omega^c})\Beurling^{m-1}(\chi_\Omega \, g)\right)- \Beurling \left(\chi_\Omega\left((\Beurling_\Omega)^{m-1}g\right)\right)\right)\\
	& =\chi_\Omega \Beurling\left(\chi_{\Omega^c} \Beurling^{m-1}(\chi_\Omega g)\right)+ \chi_\Omega \Beurling \left(\chi_\Omega \left(\Beurling^{m-1}(\chi_\Omega\cdot)-(\Beurling_\Omega(\cdot))^{m-1}\right)g\right) =\mathcal{R}_m g + \Beurling_\Omega \circ A^{(2)}_{m-1}g.
\end{align*}
Note that by definition 
\begin{equation}\label{eqAmIsBounded}
\mathcal{R}_m=\left(A^{(2)}_m -\Beurling_\Omega \circ A^{(2)}_{m-1}\right)
\end{equation} 
 is bounded in $W^{n,p}(\Omega)$. In Section \ref{secCompactnessDeath} we will prove the compactness of $\mathcal{R}_m$, which, by induction, will prove the compactness of $A^{(2)}_m$.
\begin{lemma}\label{lemCompactnessDeath}
For every $m$, the operator $\mathcal{R}_m$ is compact in $W^{n,p}(\Omega)$.
\end{lemma}

Now, the following claim is the remaining ingredient for the proof of Theorem \ref{theoInvertBeltrami}.
\begin{claim}\label{claimInvertible}
 For $m$ large enough, $A^{(1)}_m $ is invertible.
\end{claim}
\begin{proof}
Since $p>2$ we can use Theorem \ref{theoRunstSickel} to conclude that 
for every $g\in W^{n,p}(\Omega)$
\begin{align*}
\norm{\mu^m (\Beurling^m)_\Omega g}_{W^{n,p}(\Omega)}
	& \lesssim \norm{\mu^m}_{W^{n,p}(\Omega)}\norm{(\Beurling^{m})_\Omega g}_{W^{n,p}(\Omega)}\\
	& \lesssim m^{n} \norm{\mu}^{m-n}_{L^\infty}\norm{\mu}_{W^{n,p}(\Omega)}^n \norm{(\Beurling^m)_\Omega}_{W^{n,p}(\Omega)\to W^{n,p}(\Omega)}\norm{g}_{W^{n,p}(\Omega)}.
\end{align*}

By Theorem \ref{theoGeometricPGtr2}, for any $\epsilon>0$ there are constants depending on the Lipschitz character of $\Omega$ (and other parameters) but not on $m$, such that 
$$\norm{(\Beurling^m)_\Omega}_{W^{n,p}(\Omega)\to W^{n,p}(\Omega)} \lesssim m^{n+1}\left((1+\epsilon)^{m} + \norm{N}_{B^{n-1/p}_{p,p}(\partial\Omega)}\right).$$
In particular, if we choose $1+\epsilon<\frac1{\norm{\mu}_\infty}$, we get that for $m$ large enough, the operator norm $\norm{\mu^m (\Beurling^m)_\Omega}_{W^{n,p}(\Omega)\to W^{n,p}(\Omega)}<1$ and, thus, $A^{(1)}_m$ in \rf{eqFactorizePm} is invertible.
\end{proof}

\begin{proof}[Proof of Theorem \ref{theoInvertBeltrami}]
Putting together Lemmas \ref{lemCompactness} and \ref{lemCompactnessDeath}, Claim \ref{claimInvertible}, and \rf{eqFactorizePm}, we get that $I_\Omega-(\mu \Beurling_\Omega)^m$ can be expressed as the sum of an invertible operator and a compact one for $m$ big enough and, by \rf{eqPmToIMu}, we can deduce that $I_\Omega - \mu \Beurling_\Omega$ is a  Fredholm operator (see \cite[Theorem 5.5]{Schechter}). The same argument works with any other operator $I_\Omega-t \mu \Beurling_\Omega$ for $0<t<1/\norm{\mu}_\infty$. It is well known that the Fredholm index is continuous with respect to the operator norm on Fredholm operators (see \cite[Theorem 5.11]{Schechter}), so the index of $I_\Omega-\mu \Beurling_\Omega$ must be the same index of $I_\Omega$, that is, $0$.

It only remains to see that our operator is injective to prove that it is invertible. Since $\mu$ is continuous, by \cite{Iwaniec} the operator $I-\mu \Beurling$ is injective in $L^p$. Thus, if $g\in W^{n,p}(\Omega)$, and $(I_\Omega - \mu \Beurling_\Omega)g=0$, we define $G(z)= g(z)$ if $z\in\Omega$ and $G(z)=0$ otherwise, and then we have that 
$$(I - \mu \Beurling)G=(I - \mu \chi_\Omega \Beurling)(\chi_\Omega G)=(I_\Omega - \mu \Beurling_\Omega)g=0.$$
 By the injectivity of the former, we get that $G=0$ and, thus, $g=0$ as a function of $W^{n,p}(\Omega)$.
 
Now, remember that the principal solution of \rf{eqBeltrami} is $f(z)=\Cauchy h(z)+z$, where
$$h=(I-\mu \Beurling)^{-1} \mu,$$
that is, $h + \mu \Beurling(h)=\mu$, so $\supp(h)\subset \supp(\mu)\subset \overline{\Omega}$ and, thus, $\chi_\Omega h + \mu \Beurling_\Omega(h)=h + \mu \Beurling(h)=\mu$ modulo null sets, so 
 $$h|_\Omega=(I_\Omega-\mu \Beurling_\Omega)^{-1} \mu,$$
proving that $h\in W^{n,p}(\Omega)$. By Theorem \ref{theoCauchy} we have that $\Cauchy h\in L^p(\C)$. Since the  derivatives of the principal solution, $\overline{\partial}f=h$ and $\partial f = \Beurling h + 1=\Beurling_\Omega h+\chi_{\Omega^c} \Beurling h + 1$, are in $W^{n,p}(\Omega)$, we have $f\in W^{n+1,p}(\Omega)$.
\end{proof}

\subsection{Compactness of the commutator}\label{secCompactness}

\begin{proof}[Proof of Lemma \ref{lemCompactness}]
We want to see that for any $\mu\in W^{n,p}(\Omega)\cap L^\infty$, the commutator $[\mu, \Beurling_\Omega]$ is compact. The idea is to show that it has a regularizing kernel. In particular, we will prove that assuming some extra condition on the regularity of $\mu$, then the commutator maps $W^{n,p}(\Omega)$ to $W^{n+1,p}(\Omega)$.  This will imply the compactness of the commutator as a self-map of $W^{n,p}(\Omega)$ and, by a classical argument on approximation of operators, this will be extended to any given $\mu$.

First we will see that we can assume $\mu$ to be $C^\infty_c(\C)$ without loss of generality by an approxi-mation procedure. Indeed, since $\Omega$ is an extension domain, for every $\mu\in W^{n,p}(\Omega)$, there is a function $E\mu$ with $\norm{E\mu}_{W^{n,p}(\C)}\leq C \norm{\mu}_{W^{n,p}(\Omega)}$ such that $E\mu|_\Omega=\mu\chi_\Omega$. Now, $E \mu$ can be approximated by a sequence of functions $\{{\mu}_j\}_{j\in\N}\subset C^\infty_c(\C)$ in $W^{n,p}(\C)$ and one can define the operator $[\mu_j, \Beurling_\Omega]:W^{n,p}(\Omega)\to W^{n,p}(\Omega)$. Since $W^{n,p}(\Omega)$ is a multiplicative algebra, one can check that $\{[\mu_j, \Beurling_\Omega]\}_{j\in\N}$ is a sequence of operators converging to $[\mu, \Beurling_\Omega]$ in the operator norm. Thus, it is enough to prove that the operators $[\mu_j, \Beurling_\Omega]$  are compact in $W^{n,p}(\Omega)$ for all $j$ (see \cite[Theorem 4.11]{Schechter}).

Let $\mu$ be a $C^\infty_c(\C)$ function. We will prove that the commutator $[\mu,\Beurling_\Omega]$ is a smoothing operator, mapping $W^{n,p}(\Omega)$ into $W^{n+1,p}(\Omega)$. Consider $f\in W^{n,p}(\Omega)$, a Whitney covering $\mathcal{W}$ with appropriate constants and, for every $Q\in\mathcal{W}$, choose a bump function $\chi_{\frac32 Q}\leq \varphi_Q\leq \chi_{2Q}$ with $\norm{\nabla^j\varphi_Q}_{L^\infty}\lesssim\frac{C_j}{\ell(Q)^j}$. Recall that we defined $\mathbf{P}_{3Q}^{n-1} f$ to be the approximating polynomial of $f$ around $3Q$. Then, we split the norm in three terms,
\begin{align}\label{eqCommutatorF3Terms}
\norm{\nabla^{n+1}[\mu, \Beurling_\Omega]f}_{L^p(\Omega)}^p
			& \lesssim_p \sum_{Q\in\mathcal{W}}\norm{\nabla^{n+1}[\mu, \Beurling_\Omega]\left(\left(f-\mathbf{P}_{3Q}^{n-1} f\right)\varphi_Q\right)}_{L^p(Q)}^p\\
\nonumber	& \quad +\sum_{Q\in\mathcal{W}}\norm{\nabla^{n+1}[\mu, \Beurling_\Omega] \left(\left(f-\mathbf{P}_{3Q}^{n-1} f\right)(\chi_\Omega-\varphi_Q)\right)}_{L^p(Q)}^p \\
\nonumber	& \quad + \sum_{Q\in\mathcal{W}}\norm{\nabla^{n+1}[\mu, \Beurling_\Omega]\left(\mathbf{P}_{3Q}^{n-1} f\right)}_{L^p(Q)}^p =: \circled{1}+\circled{2}+\circled{3}.
\end{align}
			
First we study $\circled{1}$. In this case, we can use the following classical trick for compactly supported functions. Given $\varphi\in C^\infty_c(\C)$ and $g\in L^p$, then $\Cauchy g\in W^{1,p}(\supp(\varphi))$ by \rf{eqCauchyLocalized}. Therefore, we can use Leibniz' rule \rf{eqLeibniz} for the first order derivatives of $ \varphi \cdot \Cauchy g$ (see \cite[Section 5.2.3]{Evans}), and by Theorem \ref{theoCauchy} we get
\begin{align}\label{eqCommutesForSmooth}
\varphi \cdot \Beurling(g)- \Beurling(\varphi \cdot g)
\nonumber	& =\varphi \cdot \partial \Cauchy g - \Beurling(\varphi \cdot \overline{\partial} \Cauchy g)=   -\partial \varphi \cdot \Cauchy g +\partial (\varphi \cdot \Cauchy g)- \overline{\partial}\Beurling(\varphi \cdot \Cauchy g) + \Beurling(\overline{\partial}\varphi \cdot \Cauchy g)\\
			& =\Beurling(\overline{\partial}\varphi \cdot \Cauchy g) -\partial\varphi \cdot \Cauchy g .
\end{align}

Thus, for a fixed cube $Q$, since we assumed that $\mu\in C^\infty_c(\C)$, we have that
\begin{equation*}
[\mu, \Beurling]\left(\left(f-\mathbf{P}_{3Q}^{n-1} f\right)\varphi_Q\right)= \Beurling\left(\overline{\partial}\mu \cdot \Cauchy \left(\left(f-\mathbf{P}_{3Q}^{n-1} f\right)\varphi_Q\right)\right) -\partial\mu \cdot \Cauchy \left(\left(f-\mathbf{P}_{3Q}^{n-1} f\right)\varphi_Q\right).
\end{equation*}
Therefore, using the boundedness of the Beurling transform and the fact that it commutes with derivatives, we have that
\begin{align*}
\circled{1}
	& = \sum_{Q}\norm{\nabla^{n+1}[\mu, \Beurling]\left(\left(f-\mathbf{P}_{3Q}^{n-1} f\right)\varphi_Q\right)}_{L^p(Q)}^p\\
	& \lesssim_p \sum_{Q}\norm{\nabla^{n+1}\left(\overline \partial \mu\cdot  \Cauchy \left(\left(f-\mathbf{P}_{3Q}^{n-1} f\right)\varphi_Q\right) \right)}_{L^p}^p+ \sum_{Q}\norm{\nabla^{n+1}\left(\partial \mu\cdot \Cauchy \left(\left(f-\mathbf{P}_{3Q}^{n-1} f\right)\varphi_Q\right) \right)}_{L^p}^p\\
	& \leq \sum_{Q}\sum_{j=0}^{n+1}\norm{\mu}_{W^{n+2,\infty}}^p\norm{\nabla^{j} \Cauchy \left(\left(f-\mathbf{P}_{3Q}^{n-1} f\right)\varphi_Q\right)}_{L^p}^p
\end{align*}
and, using the identities $\partial \Cauchy =\Beurling$, $\overline \partial \Cauchy =Id$ (when $j>0$ in the previous sum) together with \rf{eqCauchyCompact} from Theorem \ref{theoCauchy} (when $j=0$) we can estimate
\begin{align*}
\circled{1}
	& \lesssim_p \norm{\mu}_{W^{n+2,\infty}}^p \sum_{Q}\left(\sum_{j=1}^{n+1}\norm{\nabla^{j-1}\left(\left(f-\mathbf{P}_{3Q}^{n-1} f\right)\varphi_Q\right)}_{L^p(2Q)}^p
		+ \ell(Q)^p\norm{f-\mathbf{P}_{3Q}^{n-1} f}_{L^p(2Q)}^p\right)
\end{align*}
 and, by the Poincar\'e inequality \rf{eqPoincare} we get
 \begin{align*}
\circled{1}
	& \lesssim_{n,p} \norm{\mu}_{W^{n+2,\infty}}^p \sum_{Q}\sum_{j=0}^{n+1} \ell(Q)^{(n+1-j)p}\norm{\nabla^{n}f}_{L^p(2Q)}^p
		\lesssim_{n,\Omega} \norm{\mu}_{W^{n+2,\infty}}^p \norm{\nabla^{n}f}_{L^p(\Omega)}^p.
\end{align*}

Second, we bound $\circled{2}$. Let $Q$ be a Whitney cube, let $z\in Q$ and let $\alpha\in \N^2$ with $|\alpha|=n+1$. Then, if we call
$$K_\mu(z,w)=\frac{\mu(z)-\mu(w)}{(z-w)^2},$$
then, since $z$ is not in the support of $\left(f-\mathbf{P}_{3Q}^{n-1} f\right)(\chi_\Omega-\varphi_Q)$, we have that
$$D^\alpha [\mu, \Beurling_\Omega]\left(\left(f-\mathbf{P}_{3Q}^{n-1} f\right)(\chi_\Omega-\varphi_Q)\right)(z)=\int_{\Omega}D^\alpha_z K_\mu(z,w)\left(f(w)-\mathbf{P}_{3Q}^{n-1} f(w)\right)(1-\varphi_Q(w))\, dm.$$
Note that
$$D^\alpha_z K_\mu(z,w)=(\mu(z)-\mu(w))D^\alpha_z \frac{1}{(z-w)^{2}} + \sum_{\gamma<\alpha}{\alpha \choose \gamma}D^{\alpha-\gamma}\mu(z) D^{\gamma}_z \frac{1}{(z-w)^{2}},$$
so using $|\mu(z)-\mu(w)|\leq \norm{\nabla\mu}_{L^\infty}|z-w|$ we get
$$|D^\alpha_z K_\mu(z,w)|\leq C_{n,\Omega}\norm{\mu}_{W^{n+1,\infty}} \frac{1}{|z-w|^{n+2}}.$$
 Using the duality expression of the $L^p$ norm, estimate \rf{eqChain} and Lemma \ref{lemTwoFoldedFunctionGuay}  we get
\begin{align*}
\circled{2}^\frac{1}{p}
	& \lesssim \norm{\mu}_{W^{n+1,\infty}}  \sup_{\norm{g}_{L^{p'}}\leq 1} \sum_{Q,S} \frac{\int_S\left|f(w)-\mathbf{P}_{3Q}^{n-1} f(w)\right|\, dm(w)}{\Dist(Q,S)^{n+2}} \int_Q \left|g(z)\right| \, dm(z)\\
	& \lesssim \norm{\mu}_{W^{n+1,\infty}}  \sup_{\norm{g}_{L^{p'}}\leq 1} \sum_{Q,S} \sum_{P\in[Q,S]} \frac{\norm{g}_{L^1(Q)} \norm{\nabla^n f}_{L^1(3P)} \Dist(P,S)^{n-1}\ell(S)^2}{\ell(P) \Dist(Q,S)^{n+2}} \\
	& \lesssim_{n,\Omega} \norm{\mu}_{W^{n+1,\infty}} \norm{\nabla^{n}f}_{L^p(\Omega)}.
\end{align*}
Next we use a $T(1)$ argument reducing $\circled{3}$ to the boundedness of $[\mu,\Beurling_\Omega](1)$. Consider the monomials $P_{Q,\gamma}(z):=(z-z_Q)^ \gamma$ where $z_Q$ stands for the center of $Q$. The Taylor expansion \rf{eqTaylorExpansion} of $\mathbf{P}_{3Q}^{n-1} f$ around $z_Q$  can be written as $\mathbf{P}_{3Q}^{n-1} f(z)=\sum_{|\gamma|<n}m_{Q,\gamma}P_{Q,\gamma}(z)$. Thus, we have that
$$-\pi[\mu,\Beurling_\Omega]\mathbf{P}_{3Q}^{n-1} f(z)
= \left[\mu,T^{(-2,0)}_\Omega\right]\mathbf{P}_{3Q}^{n-1} f(z)
=\sum_{|\gamma|<n} m_{Q,\gamma} \left[\mu,T^{(-2,0)}_\Omega\right] \left(P_{Q,\gamma}\right)(z),$$
and using the binomial expansion $(w-z_Q)^\gamma=\sum_{\lambda\leq\gamma} (-1)^\lambda{\gamma\choose\lambda}(z-w)^\lambda(z-z_Q)^{\gamma-\lambda}$ we have
\begin{align}\label{eqCommutatorBinomial}
-\pi[\mu,\Beurling_\Omega]\mathbf{P}_{3Q}^{n-1} f(z)
			& =\sum_{|\gamma|<n}m_{Q,\gamma} \sum_{\vec{0}\leq\lambda\leq\gamma}(-1)^\lambda{\gamma\choose\lambda}\left[\mu,T^{(-2,0)+\lambda}_\Omega\right](1)(z)\, P_{Q,\gamma-\lambda}(z),
\end{align}
that is,
\begin{align*}
\circled{3}
	& =\sum_{Q\in\mathcal{W}}\norm{\nabla^{n+1}[\mu, \Beurling_\Omega](\mathbf{P}_{3Q}^{n-1} f)}_{L^p(Q)}^p\\
	& \lesssim \sum_{|\gamma|<n}\sum_{\vec{0}\leq\lambda\leq\gamma} \sum_{Q\in\mathcal{W}}|m_{Q,\gamma}|^p \norm{\nabla^{n+1}\left(\left[\mu,T^{(-2,0)+\lambda}_\Omega\right](1)\cdot  P_{Q,\gamma-\lambda}\right)}_{L^p(Q)}^p.
\end{align*}
But every coefficient $|m_{Q,\gamma}|$ is bounded by $C\norm{f}_{W^{n-1,\infty}(Q)}$ by \rf{eqMGammaBounded} and all the derivatives of $P_{Q,\gamma}$ are uniformly bounded in $\Omega$. Therefore, we have that
\begin{align*}
\circled{3}
	& \lesssim \norm{f}_{W^{n-1,\infty}(\Omega)}^p\sum_{Q\in\mathcal{W}}\sum_{0\leq|\lambda|<n}  \norm{\left[\mu,T^{(-2,0)+\lambda}_\Omega\right]1}_{W^{n+1,p}(Q)}^p.
\end{align*}
Using the Sobolev Embedding Theorem, we get
\begin{align*}
\circled{3}
	& \lesssim \norm{f}_{W^{n,p}(\Omega)}^p\left(\sum_{0<|\lambda|<n}  \norm{\left[\mu,T^{(-2,0)+\lambda}_\Omega\right]1}_{W^{n+1,p}(\Omega)}^p + \sum_{Q\in\mathcal{W}}  \norm{\left[\mu,T^{(-2,0)}_\Omega\right]1}_{W^{n+1,p}(Q)}^p\right).
\end{align*}
Note that if $\lambda>\vec 0$, then the operator $T^{(-2,0)+\lambda}_\Omega$ has homogeneity $-2+\lambda_1+\lambda_2> - 2$ and, therefore, by Theorem \ref{theoGeometricPGtr2}, $T^{(-2,0)+\lambda}_\Omega: W^{n,p}(\Omega) \to W^{n+1,p}(\Omega)$ is bounded and, since $p>2$ and $W^{n+1,p}(\Omega)$ is a multiplicative algebra, we have that $\norm{\mu T^{(-2,0)+\lambda}_\Omega 1}_{W^{n+1,p}(\Omega)}^p +\norm{T^{(-2,0)+\lambda}_\Omega \mu}_{W^{n+1,p}(\Omega)}^p \lesssim_{n,p,\Omega} \norm{\mu}_{W^{n+1,p}(\Omega)}^p$. Therefore,
\begin{align*}
\circled{3}
	& \lesssim  \left(\norm{\mu}_{W^{n+1,p}(\Omega)}^p+  \norm{[\mu,\Beurling_\Omega](1)}_{W^{n+1,p}(\Omega)}^p\right) \norm{f}_{W^{n,p}(\Omega)}^p,
\end{align*}
so we have reduced the proof of Lemma \ref{lemCompactness} to the following claim.

\begin{claim}\label{claimCommutatorWnplus1p}
Let $2<p<\infty$, $n\in \N$. Given a bounded Lipschitz domain $\Omega$ with parameterizations in $B^{n+1-1/p}_{p,p}$ and a function $\mu\in C^\infty_c(\C)$, then $[\mu, \Beurling_\Omega](1)\in W^{n+1,p}(\Omega)$. 
\end{claim}

We know that  $[\mu, \Beurling_\Omega](1)=\mu \Beurling_\Omega (1)- \Beurling_\Omega (\mu) \in W^{n,p}(\Omega)$. We want to prove that $\nabla^{n+1}[\mu, \Beurling_\Omega]1\in L^p$. To do so, we split the norm in the same spirit of \rf{eqCommutatorF3Terms}, but chopping $\mu$ instead of $f$:
\begin{align*}
\norm{\nabla^{n+1}[\mu, \Beurling_\Omega](1)}_{L^p(\Omega)}^p
			& \lesssim_p \sum_{Q\in\mathcal{W}}\norm{\nabla^{n+1}\left[\left(\mu-\mathbf{P}^{n+2}_{3Q} \mu\right)\varphi_Q, \Beurling_\Omega\right](1)}_{L^p(Q)}^p\\
	& \quad +\sum_{Q\in\mathcal{W}}\norm{\nabla^{n+1}\left[\left(\mu-\mathbf{P}^{n+2}_{3Q} \mu\right)(\chi_\Omega-\varphi_Q), \Beurling_\Omega\right](1)}_{L^p(Q)}^p \\
	& \quad + \sum_{Q\in\mathcal{W}}\norm{\nabla^{n+1}\left[\mathbf{P}^{n+2}_{3Q} \mu, \Beurling_\Omega\right](1)}_{L^p(Q)}^p =: \circled{4}+\circled{5}+\circled{6}.
\end{align*}

First we consider $\circled{4}$. Since $\left(\mu-\mathbf{P}^{n+2}_{3Q} \mu\right)\varphi_Q \in C^\infty_c$, by \rf{eqCommutesForSmooth} we have that
\begin{align*}
\sum_{Q}\norm{\nabla^{n+1}\left[\left(\mu-\mathbf{P}^{n+2}_{3Q} \mu\right)\varphi_Q, \Beurling \right]\chi_\Omega}_{L^p(\C)}^p 
	& \lesssim_p \sum_Q \norm{\nabla^{n+1}\left(\partial \left(\left(\mu-\mathbf{P}^{n+2}_{3Q} \mu\right)\varphi_Q\right) \cdot \Cauchy \chi_\Omega\right)}_{L^p(2Q)}^p\\
	& \quad + \sum_Q \norm{\nabla^{n+1}\left(\overline \partial \left(\left(\mu-\mathbf{P}^{n+2}_{3Q} \mu\right)\varphi_Q\right) \cdot \Cauchy \chi_\Omega\right)}_{L^p(2Q)}^p
\end{align*}
and, using Leibniz' rule \rf{eqLeibniz}, H\"older's inequality, and the finite overlapping of double Whitney cubes,
\begin{align}\label{eqCommutator1Leibniz}
\circled{4}
	& \lesssim_p \sum_{j=0}^{n+1} \left(\sup_{Q\in\mathcal{W}}\norm{\nabla^{j+1}\left(\left(\mu-\mathbf{P}^{n+2}_{3Q} \mu\right)\varphi_Q\right)}_{L^\infty(2Q)}^p\right) \cdot\norm{\nabla^{n+1-j}\Cauchy \chi_\Omega}_{L^p(\Omega)}^p.
\end{align}

To bound $\circled{4}$ it remains to see that $\sup_{Q\in\mathcal{W}}\norm{\nabla^{j+1}\left(\left(\mu-\mathbf{P}^{n+2}_{3Q} \mu\right)\varphi_Q\right)}_{L^\infty(2Q)}^p<\infty$. The Poincar\'e inequality \rf{eqPoincare} leads to 
\begin{align}\label{eqCommutator1Poincare}
\norm{\nabla^{j+1}\left(\left(\mu-\mathbf{P}^{n+2}_{3Q} \mu\right)\varphi_Q\right)}_{L^\infty(2Q)}^p
			& \lesssim \norm{\nabla^{n+3}\mu}_{L^{\infty}(3Q)}^p.
\end{align}
Thus, the bounds \rf{eqCommutator1Leibniz} and \rf{eqCommutator1Poincare} yield 
\begin{equation*}
\circled{4}\leq C_{p,n,\diam\Omega}\norm{\nabla^{n+3}\mu}_{L^{\infty}(\Omega)}^p \norm{\Cauchy \chi_\Omega}_{W^{n+1,p}(\Omega)}^p,
\end{equation*}
which is finite by Theorem \ref{theoGeometricPGtr2}.

Next we face $\circled{5}$. Note that for a given Whitney cube $Q$, if $z\in Q$, then $\chi_\Omega(z)-\varphi_Q(z)=0$, so 
$$\circled{5}=\sum_{Q\in\mathcal{W}}\norm{\nabla^{n+1} \Beurling\left(\left(\mu-\mathbf{P}^{n+2}_{3Q} \mu\right)(\chi_\Omega-\varphi_Q)\right)}_{L^p(Q)}^p .$$
Moreover, for $z\in Q\in\mathcal{W}$, we have 
$$\partial^{n+1} \Beurling\left(\left(\mu-\mathbf{P}^{n+2}_{3Q} \mu\right)(\chi_\Omega-\varphi_Q)\right)(z)=c_n \int_{\Omega\setminus \frac32Q} \frac{\left(\mu(w)-\mathbf{P}^{n+2}_{3Q} \mu(w)\right)(1-\varphi_Q(w))}{(z-w)^{3+n}}dm(w).$$
Since $\overline{\partial} \Beurling \left(\left(\mu-\mathbf{P}^{n+2}_{3Q} \mu\right)(\chi_\Omega-\varphi_Q)\right)(z)=0$, only $\partial^{n+1}$ is non zero in the $(n+1)$-th gradient, so
$$\left|\nabla^{n+1} \Beurling\left(\left(\mu-\mathbf{P}^{n+2}_{3Q} \mu\right)(\chi_\Omega-\varphi_Q)\right)(z)\right|\lesssim \sum_{S\in\mathcal{W}}   \frac{1}{\Dist(Q,S)^{3+n}}\norm{\mu-\mathbf{P}^{n+2}_{3Q} \mu}_{L^1(S)}.$$
Consider an admissible chain $[S,Q]$. By \rf{eqChain} we have that
$$\norm{\mu-\mathbf{P}^{n+2}_{3Q} \mu}_{L^1(S)}\lesssim \sum_{P\in[S,Q]}\frac{\ell(S)^2 \Dist(P,S)^{n+2}}{\ell(P)}\norm{\nabla^{n+3} \mu}_{L^1(5P)}.$$
Combining all these facts with the expression of the norm by duality and Lemma \ref{lemTwoFoldedFunctionGuay}, we get
\begin{align*}
\circled{5}^{\frac1p}
	& \lesssim \sup_{g\in L^{p'} (\Omega): \norm{g}_{p'}\leq1} \sum_Q \int_Q g\, dm \sum_{S\in\mathcal{W}}   \frac{1}{\Dist(Q,S)^{3+n}}\sum_{P\in[S,Q]}\frac{\ell(S)^2 \Dist(P,S)^{n+2}}{\ell(P)}\norm{\nabla^{n+3} \mu}_{L^1(5P)}\\
	& \lesssim \diam (\Omega)^2 \sup_{g\in L^{p'} (\Omega): \norm{g}_{p'}\leq1} \sum_Q \sum_{S}   \sum_{P\in[S,Q]}\frac{\ell(S)^2\norm{\nabla^{n+3} \mu}_{L^1(5P)} \norm{g}_{L^1(Q)}}{\ell(P)\Dist(Q,S)^3} 
	 \lesssim \norm{\nabla^{n+3} \mu}_{L^p(\Omega)}.
\end{align*}

Finally we focus on 
$$\circled{6}=\sum_{Q\in\mathcal{W}}\norm{\nabla^{n+1}\left[\mathbf{P}^{n+2}_{3Q} \mu, \Beurling_\Omega\right](1)}_{L^p(Q)}^p .$$
 Consider first a monomial $P_{Q,\gamma}(z)=(z-z_Q)^ \gamma $ for a multiindex $\gamma\in \N^2$. Then, as we did in \rf{eqCommutatorBinomial}, we use the binomial expression $P_{Q,\gamma}(w)=\sum_{\lambda\leq\gamma} (-1)^{|\lambda|}{\gamma \choose \lambda} (z-w)^\lambda (z-z_Q)^{\gamma-\lambda}$ to deduce that
$$-\pi \Beurling_\Omega P_{Q,\gamma}(z)=T^{(-2,0)}_\Omega P_{Q,\gamma}(z)=\sum_{\vec{0}\leq\lambda\leq \gamma} (-1)^{|\lambda|}{\gamma \choose \lambda} T^{(-2,0)+\lambda}_\Omega(1)(z) (z-z_Q)^{\gamma-\lambda}.$$
Note that the term for $\lambda=\vec{0}$ in the right-hand side of this expression is $T^{(-2,0)}_\Omega (1)(z) P_{Q,\gamma}(z)$, so it cancels out in the commutator:
\begin{equation}\label{eqCommutatorPolyCancels}
-\pi [P_{Q,\gamma}, \Beurling_\Omega ](1)(z)=\sum_{\vec{0}<\lambda\leq \gamma} (-1)^{|\lambda|}{\gamma \choose \lambda} T^{(-2,0)+\lambda}_\Omega(1)(z) P_{Q,\gamma-\lambda}(z).
\end{equation}

Now, writting $\mathbf{P}^{n+2}_{3Q}\mu(z)=\sum_{|\gamma|\leq n+2}m_{Q,\gamma}P_{Q,\gamma}(z)$ we have that
\begin{align*}
\circled{6}
	& = \sum_{Q\in\mathcal{W}} \norm{\nabla^{n+1} [\mathbf{P}^{n+2}_{3Q}\mu, \Beurling_\Omega ](1)}_{L^p(Q)}^p\leq \sum_{Q\in\mathcal{W}}\sum_{\gamma\leq n+2} |m_{Q,\gamma}|^p \norm{\nabla^{n+1} [P_{Q,\gamma}, \Beurling_\Omega ](1)}_{L^p(Q)}^p,
\end{align*}
so using \rf{eqMGammaBounded} and \rf{eqCommutatorPolyCancels} together with Leibniz' rule \rf{eqLeibniz}, we get
\begin{align}\label{eqInequality6Final}
\circled{6}
\nonumber	& \lesssim \norm{\mu}_{W^{n+2,\infty}}^p \sum_{Q\in\mathcal{W}}\sum_{\gamma\leq n+2}\sum_{\vec{0}<\lambda\leq \gamma}  \sum_{j=0}^{n+1}\norm{\nabla^{j} T^{(-2,0)+\lambda}_\Omega(1)}_{L^p(Q)}^p\norm{ \nabla^{n+1-j}P_{Q,\gamma-\lambda}}_{L^\infty(Q)}^p\\
			& \leq C_{n,p,\Omega} \norm{\mu}_{W^{n+2,\infty}}^p \sum_{\vec{0}<\lambda: |\lambda|\leq n+2} \norm{T^{(-2,0)+\lambda}_\Omega (1)}_{W^{n+1,p}(\Omega)}^p.
\end{align}
In the last sum we have that $T^{(-2,0)+\lambda}_\Omega (1) \in W^{n+1,p}(\Omega)$ for all $\lambda >\vec{0}$ by  Theorem \ref{theoGeometricPGtr2} because the operators $T^{(-2,0)+\lambda}$ have homogeneity greater than $-2$. Thus, the right-hand side of \rf{eqInequality6Final} is finite.
\end{proof}

\subsection{Some technical details}\label{secTechnical}
Given $\vec{m}=(m_1,m_2,m_3)\in \N^3$, let us define the line integral 
\begin{equation}\label{eqKernelVectorM}
K_{\vec{m}}(z,\xi):=\int_{\partial\Omega} \frac{(\overline{w-\xi})^{m_3}}{(z-w)^{m_1}\,(w-\xi)^{m_2}}\, dw
 \end{equation}
for all $z, \xi \in \Omega$, where the path integral is oriented counterclockwise.

Given a $j$ times differentiable function $f$, we will write 
$$P^{j}_z(f)(\xi)=\sum_{|\vec{i}|\leq j}\frac{D^{\vec{i}}f(z)}{\vec{i}!} (\xi-z)^{\vec{i}}$$
for its $j$-th degree Taylor polynomial centered in the point $z$.

Mateu, Orobitg and Verdera study the kernel $K_{(2,m+1,m)}(z,\xi)$ for $m\in\N$ in \cite[Lemma 6]{MateuOrobitgVerdera} assuming the boundary of the domain $\Omega$ to be in $C^{1,\varepsilon}$ for $\varepsilon<1$. They prove the size inequality 
$$|K_{(2,m+1,m)}(z,\xi)|\lesssim \frac{1}{|z-\xi|^{2-\varepsilon}}$$
and a smoothness inequality in the same spirit. In \cite{CruzMateuOrobitg}, when dealing with the compactness of the operator $\mathcal{R}_mf =\chi_\Omega \Beurling\left(\chi_{\Omega^c} \Beurling^{m-1}(\chi_\Omega f)\right)$ on $W^{s,p}(\Omega)$ for $0<s<1$,  this is used to prove that the Beltrami coefficient $\mu \in W^{s,p}(\Omega)$ implies the principal solution of $\overline \partial f =\mu \partial f$ being in $W^{s+1,p}(\Omega)$ only for $s<\varepsilon$. These bounds are not enough for us in this form and, moreover, we will consider $m_1>2$ (this comes from differenciating the kernel of $\mathcal{R}_m$, something that we have to do in order to study the classical Sobolev spaces). Nevertheless, their argument can be adapted to the case of the boundary being in the space ${B}^{n+1-1/p}_{p,p}\subset C^{n,1-2/p}$ to get Proposition \ref{propoKernelExpression} below, which will be used to prove Lemma \ref{lemCompactnessDeath}. The proof has the same basic structure as in \cite{CruzMateuOrobitg}, but it is more sophisticated and makes use, among other tools, of a lemma of combinatorial type that is proven later on.

We will use some auxiliary functions.
\begin{definition}\label{defHm3xi}
Let us define
\begin{equation*}
H_{m_3,\xi}(w):=\frac{1}{2\pi i}\int_{\partial\Omega}\frac{(\overline{\tau-\xi})^{m_3}}{\tau-w}\, d\tau \mbox{\,\,\,\, for every }w,\xi\notin\partial\Omega,
\end{equation*}
and
\begin{equation}\label{eqhm3}
h_{m_3}(z):=\int_{\partial\Omega}\frac{(\overline{\tau-z})^{m_3}}{\tau-z} \,d\tau =2\pi iH_{m_3,z}(z) \mbox{\,\,\,\, for every }z\in\Omega.
\end{equation}
\end{definition}

\begin{proposition}\label{propoKernelExpression}
Let $\Omega$ be a bounded Lipschitz domain, and let $\vec{m}=(m_1,m_2,m_3)\in \N^3$ with $m_1\geq 3$, $m_2,m_3\geq1$ and $m_2 \leq m_1+m_3-2$. Then, the weak derivatives of order $m_3$ of $h_{m_3}$ in $\Omega$ are
\begin{equation}\label{eqDerivativesBeurlings}
\partial^{j}\overline\partial^{m_3-j}h_{m_3}=c_{m_3,j}\Beurling^{j}\chi_\Omega \mbox{, \,\,\,\, for }0\leq j\leq m_3.
\end{equation} 
Moreover, for every pair $z, \xi \in \Omega$ with $z\neq \xi$, we have that
\begin{equation}\label{eqExpressionTaylorAndMore}
K_{\vec{m}}(z,\xi)=c_{\vec{m}}{\partial^{m_1-2}}\Beurling\chi_\Omega (z) \frac{(\overline{\xi-z})^{m_3-1}}{(\xi-z)^{m_2}}+\sum_{j\leq m_2-1} \frac{c_{\vec{m},j} R_{m_1+m_3-3,j}^{m_3}(z,\xi)}{(\xi-z)^{m_2+m_1-1-j}},
\end{equation}
where
\begin{equation}\label{eqTaylorError}
R_{M,j}^{m_3}(z,\xi):=\partial^jh_{m_3}(\xi)-P^{M-j}_z(\partial^j h_{m_3})(\xi)
\end{equation}
is the Taylor error term of order $M-j$ for the function $\partial^j h_{m_3}$.
\end{proposition}

We begin by noting some remarkable properties of these functions. 
\begin{lemma}[{see \cite[p. 143]{Verdera}}]
Let $\Omega$ be  a bounded Lipschitz domain. Given $\xi\notin\partial \Omega$ and $w\in\partial\Omega$, if we write $H_{m_3,\xi}^-(w)$ for the interior non-tangential limit of $H_{m_3,\xi}(\zeta)$ when $\zeta\to w$ and $H_{m_3,\xi}^+(w)$ for the exterior one, we have the Plemelj formula
\begin{equation}\label{eqPlemelj}
(\overline{w-\xi})^{m_3}=H_{m_3,\xi}^-(w)-H_{m_3,\xi}^+(w).
\end{equation}
 \end{lemma}

\begin{remark}\label{remDerivativesVeryNice}
Given $\vec{j}=(j_1,j_2)$ with $j_2\leq m_3$, taking partial derivatives on \rf{eqhm3} we get
$$D^{\vec{j}}h_{m_3}(z)=\partial^{j_1}\overline\partial^{j_2}h_{m_3}(z)=\frac{m_3! j_1!}{(m_3-j_2)!}(-1)^{j_2}\int_{\partial\Omega}\frac{(\overline{\tau-z})^{m_3-j_2}}{(\tau-z)^{1+j_1}} \, d\tau \mbox{\,\,\,\, for every }z\in\Omega$$
and, in particular, $h_{m_3}$ is infinitely many times differentiable in $\Omega$. Therefore, by Green's formula \rf{eqGreen} and the cancellation of the integrand (see \cite[(3.2)]{PratsPlanarDomains}), for $j>0$ we have
$$D^{(j,m_3-j)}h_{m_3}(z)= c_{m_3,j}\int_{\partial\Omega}\frac{(\overline{\tau-z})^{j}}{(\tau-z)^{1+j}} \, d\tau=c_{m_3,j} \int_{\Omega\setminus B(z,\varepsilon)}\frac{(\overline{w-z})^{j-1}}{(w-z)^{j+1}} dm(w)=c_{m_3,j}\Beurling^{j}\chi_\Omega(z)$$
for $\varepsilon <\dist(z,\partial\Omega)$ and, in case $j=0$, by the Residue Theorem
$$\overline\partial^{m_3}h_{m_3}(z)=c_{m_3}\int_{\partial\Omega}\frac{1}{\tau-z} \, d\tau=c_{m_3}2\pi i \chi_\Omega(z),$$
proving \rf{eqDerivativesBeurlings}.
\end{remark}

\begin{remark}
We can also relate the derivatives of both $h_{m_3}(z)$ and $H_{m_3,\xi}(z)$ for any pair $z,\xi \in \Omega$. By Definition \ref{defHm3xi} and the previous remark, we have that
\begin{align*}
2\pi i H_{m_3,\xi}(z)
	& =\sum_{l=0}^{m_3}\int_{\partial\Omega} {m_3 \choose l} \frac{(\overline{\tau-z})^{m_3-l}(\overline{z-\xi})^{l}}{\tau-z}\, d\tau\\
	& =\sum_{l=0}^{m_3} \frac{m_3!}{(m_3-l)! l!} \overline\partial^l h_{m_3}(z)\frac{(m_3-l)!}{m_3!} (-1)^l  (\overline{\xi-z})^{l} (-1)^l,
 \end{align*}
 that is,
\begin{align}\label{eqRelateHm3Andhm3Derivatives}
2\pi i \partial^j H_{m_3,\xi}(z)
	& =\sum_{l=0}^{m_3} \frac{1}{l!} D^{(j,l)} h_{m_3}(z)   (\overline{\xi-z})^{l}.
 \end{align}
 \end{remark}
 
\begin{proof}[Proof of Proposition \ref{propoKernelExpression}]
Since \rf{eqDerivativesBeurlings} is shown in Remark \ref{remDerivativesVeryNice}, it remains to prove \rf{eqExpressionTaylorAndMore}.

Consider $z,\xi\in \Omega$. Then $\frac{H_{m_3,\xi}(w)}{(z-w)^{m_1}\,(w-\xi)^{m_2}}$ decays at $\infty$ as $\frac{1}{|w|^{m_1+m_2+1}}$ and it is holomorphic in $\Omega^c$. Thus, by Cauchy's theorem and a  limiting argument we have that
$$K_{\vec{m}}(z,\xi)=\int_{\partial\Omega} \frac{(\overline{w-\xi})^{m_3}}{(z-w)^{m_1}\,(w-\xi)^{m_2}}\, dw=\int_{\partial\Omega} \frac{(\overline{w-\xi})^{m_3}+H^+_{m_3,\xi}(w)}{(z-w)^{m_1}\,(w-\xi)^{m_2}}\, dw,$$
 and using \rf{eqPlemelj}, 
$$K_{\vec{m}}(z,\xi)=(-1)^{m_1}\int_{\partial\Omega} \frac{H_{m_3,\xi}^-(w)}{(w-z)^{m_1}\,(w-\xi)^{m_2}}\, dw.$$

Note that  $H_{m_3,\xi}(w)$ is holomorphic in $\Omega$, implying that the integrand above is meromorphic in $\Omega$ with poles in $z$ and $\xi$. Moreover, $H_{m_3,\xi}^-\in L^2(\partial\Omega)$ by the boundedness of the Cauchy transform in $L^2(\Gamma)$ on a Lipschitz graph $\Gamma$ (see \cite{Verdera}, for instance). Thus, combining the Dominated Convergence Theorem and the Residue Theorem, we get
\begin{align*}
(-1)^{m_1}K_{\vec{m}}(z,\xi)
	& =2\pi i\left\{\frac{1}{(m_1-1)!}\partial^{m_1-1}\left[\frac{H_{m_3,\xi}(\cdot)}{(\cdot-\xi)^{m_2}}\right](z)+\frac{1}{(m_2-1)!}\partial^{m_2-1}\left[\frac{H_{m_3,\xi}(\cdot)}{(\cdot-z)^{m_1}}\right](\xi)\right\}.
\end{align*}
Therefore,
\begin{align*}
\frac{(-1)^{m_1}}{2\pi i}K_{\vec{m}}(z,\xi)
	& = \frac{1}{(m_1-1)!}\sum_{\substack{j_1,j_2\geq 0\\j_1+j_2=m_1-1}}\frac{(m_1-1)!}{j_1! j_2!} \frac{\partial^{j_2}H_{m_3,\xi}(z)}{(z-\xi)^{m_2+j_1}}(-1)^{j_1}\frac{(m_2+j_1-1)!}{(m_2-1)!}\\
	& \quad + \frac{1}{(m_2-1)!}\sum_{\substack{j_1,j_2\geq 0\\j_1+j_2=m_2-1}}\frac{(m_2-1)!}{j_1! j_2!} \frac{\partial^{j_2}H_{m_3,\xi}(\xi)}{(\xi-z)^{m_1+j_1}}(-1)^{j_1}\frac{(m_1+j_1-1)!}{(m_1-1)!} .
\end{align*}

Simplifying and using \rf{eqRelateHm3Andhm3Derivatives} on the first sum of the right-hand side and  \rf{eqhm3} on the second one, we get 
\begin{align}\label{eqSolvingH}
(-1)^{m_1+m_2}K_{\vec{m}}(z,\xi)
\nonumber	& = \sum_{\substack{j_1,j_2\geq 0\\j_1+j_2=m_1-1}} {m_2+j_1-1 \choose m_2-1}\frac{1}{j_2!} \frac{1}{(\xi-z)^{m_2+j_1}}\sum_{l=0}^{m_3} \frac{1}{l!} D^{(j_2,l)} h_{m_3}(z)   (\overline{\xi-z})^{l}\\
			& \quad +\sum_{\substack{j_1,j_2\geq 0\\j_1+j_2=m_2-1}}{m_1+j_1-1 \choose m_1-1}\frac{1}{j_2!} \frac{\partial^{j_2}h_{m_3}(\xi)}{(\xi-z)^{m_1+j_1}}(-1)^{j_2+1}.
\end{align}
The key idea for the rest of the proof is that the first term in the right-hand side of \rf{eqSolvingH} contains the Taylor expansion of the functions in the second one. 

Let $m_2-1\leq M\leq m_1+m_3-2$ (later on we will actually use the value $M=m_1+m_3-3$). Then, using the Taylor approximating polynomial of each $\partial^{j_2}h_{m_3}$ and multiplying by $(\xi-z)^{m_1+m_2-1}$ we get
\begin{align*}
-K_{\vec{m}}(z,\xi)(z-\xi)^{m_1+m_2-1}
	& = \sum_{j=0}^{m_1-1}{m_2+m_1-2-j \choose m_2-1}\frac{1}{j!} \sum_{l=0}^{m_3} \frac{1}{l!} D^{(j,l)} h_{m_3}(z)   ({\xi-z})^{(j,l)}\\
	& \quad - \sum_{j=0}^{m_2-1}{m_1+m_2-2-j \choose m_1-1}\frac{(-1)^{j}}{j!} (\xi-z)^{j} R_{M,j}^{m_3}(z,\xi)\\
	& \quad - \sum_{j=0}^{m_2-1}{m_1+m_2-2-j \choose m_1-1}\frac{(-1)^{j}}{j!}  \sum_{|\vec{i}|\leq M-j}\frac{D^{\vec{i}}\partial^{j}h_{m_3}(z)}{\vec{i}} (\xi-z)^{\vec{i}+(j,0)} .
\end{align*}

To simplify notation, let us define the error
\begin{equation}\label{eqErrorDefinition}
E_M=-K_{\vec{m}}(z,\xi)(z-\xi)^{m_1+m_2-1}+\sum_{j=0}^{m_2-1} {m_1+m_2-2 -j \choose m_1-1}\frac{(-1)^j}{j !} (\xi-z)^{j }R_{M,j}^{m_3}(z,\xi).
\end{equation}
Then,
\begin{align*}
E_M
	& = \sum_{\substack{\alpha \geq \vec{0}\\ \alpha\leq (m_1-1,m_3)}}{m_1+m_2-2-\alpha_1 \choose m_2-1}\frac{D^{\alpha} h_{m_3}(z) }{\alpha!}   ({\xi-z})^{\alpha}\\
	& \quad - \sum_{\substack{\alpha \geq \vec{0}\\ |\alpha|\leq M}}\sum_{0\leq j\leq \min\{m_2-1,\alpha_1\}}{m_1+m_2-2-j \choose m_1-1}\frac{(-1)^{j}}{j!} \frac{D^{\alpha} h_{m_3}(z)}{(\alpha_1-j)!\alpha_2!} (\xi-z)^{\alpha} .
\end{align*}
 Note that if $\alpha_2>m_3$, we have that 
 \begin{equation}\label{eqAlphaM3}
D^\alpha h_{m_3}(z)=0
\end{equation}
(apply \rf{eqDerivativesBeurlings} with $j=0$). The same happens for the case $\alpha=(\alpha_1,m_3)$ with $\alpha_1>0$. On the other hand, if $\alpha_1>m_1-1$, then ${m_1+m_2-2-\alpha_1 \choose m_2-1}=0$. By the same token, if $j>m_2-1$, ${m_1+m_2-2-j \choose m_1-1}=0$. Thus, we can write
\begin{multline}
E_M
\nonumber	 = \sum_{|\alpha|\leq m_1+m_3-2} \frac{D^{\alpha} h_{m_3}(z) }{\alpha!}   ({\xi-z})^{\alpha} \\
			 \quad \cdot \left[{m_1+m_2-2-\alpha_1 \choose m_2-1}  - \chi_{|\alpha|\leq M} \sum_{j\leq \alpha_1}(-1)^j {m_1+m_2-2-j \choose m_1-1} {\alpha_1 \choose j} \right].
\end{multline}
Note that we have added many null terms in the previous expression, but now the proof of the proposition is reduced to Claim \ref{claimCombinatorial} below which implies that
\begin{align*}
E_M
	& = \sum_{M<|\alpha|\leq m_1+m_3-2}  {m_1+m_2-2-\alpha_1 \choose m_2-1}  \frac{D^{\alpha} h_{m_3}(z) }{\alpha!}   ({\xi-z})^{\alpha}.
\end{align*}
Taking $M=m_1+m_3-3$ in this expression, only the terms with $|\alpha|=m_1+m_3-2$ remain and, arguing as before, if $\alpha_1>m_1-1$ then ${m_1+m_2-2-\alpha_1 \choose m_2-1} =0$  and 
if $\alpha_2\geq m_3$ then \rf{eqAlphaM3} holds (in case of equality, we can use that $|\alpha|>M\geq m_3$ because we assume that $m_1\geq 3$, granting that a derivative of the characteristic function is taken here). Summing up, by \rf{eqDerivativesBeurlings} we have that
$$E_{m_1+m_3-3}=\frac{D^{(m_1-1,m_3-1)} h_{m_3}(z) }{(m_1-1)!(m_3-1)!} ({\xi-z})^{(m_1-1,m_3-1)}= c_{\vec{m}} \partial^{m_1-2} \Beurling\chi_\Omega (z)  ({\xi-z})^{(m_1-1,m_3-1)}.$$
By \rf{eqErrorDefinition} this implies \rf{eqExpressionTaylorAndMore}.
\end{proof}

\begin{claim}\label{claimCombinatorial}For any natural numbers $m_1$, $m_2$ and $\alpha_1$ we have that
$${m_1+m_2-2-\alpha_1 \choose m_2-1}  = \sum_{j=0}^{\alpha_1}(-1)^j{\alpha_1 \choose j} {m_2+m_1-2-j \choose m_1-1} .$$
\end{claim}
\begin{proof}
We have the trivial identity
$${m_1+m_2-2-\alpha_1 \choose m_2-1}={m_1+m_2-2-\alpha_1 \choose m_1-1-\alpha_1}  =\sum_{i= 0}^0(-1)^i {0\choose i}{m_1+m_2-2-\alpha_1-i\choose m_1-1-\alpha_1}.$$
Let $\kappa_1, \kappa_2, \kappa_3\in \Z$ with $\kappa_1\geq0$. We have that 
\begin{align*}
\sum_{i=0}^{\kappa_1}(-1)^i {\kappa_1 \choose i}{\kappa_3-i \choose \kappa_2}
	& = \sum_{i=0}^{\kappa_1}(-1)^i \left[{\kappa_1 \choose i}{\kappa_3+1-i \choose \kappa_2+1}-{\kappa_1 \choose i}{\kappa_3-i \choose \kappa_2+1}\right]\\
	& = \sum_{j=0}^{\kappa_1+1}(-1)^j\left[{\kappa_1 \choose j}{\kappa_3+1-j \choose \kappa_2+1}+{\kappa_1 \choose j-1}{\kappa_3+1-j \choose \kappa_2+1}\right]\\
	& = \sum_{j=0}^{\kappa_1+1}(-1)^j{\kappa_1+1 \choose j}{\kappa_3+1-j \choose \kappa_2+1}.
\end{align*}

Arguing by induction we get that
$$\sum_{i=0}^{0}(-1)^i{0 \choose i} {m_1+m_2-2-\alpha_1-i\choose m_1-1-\alpha_1}=\dots=  \sum_{j=0}^{ \alpha_1}(-1)^j{\alpha_1 \choose j} {m_2+m_1-2-j \choose m_1-1} .$$
\end{proof}


\begin{lemma}\label{lemTaylor}
Let $z, \xi$ be two points in an extension domain $\Omega\subset \R^d$, $M\geq 1$ a natural number, $p>d$ and $f\in W^{M+1,p}(\Omega)$. Then, writing $\sigma:={\sigma_{d,p}}=1-\frac{d}p$, the Taylor error term satisfies the estimate 
$$\left|f(\xi)-P^{M}_zf(\xi)\right|\leq C \norm{f}_{W^{M+1,p}(\Omega)} |z-\xi|^{M+\sigma}.$$
\end{lemma}
\begin{proof}
Let us assume that $0\in\Omega$. Using the extension $E:W^{M+s,p}(\Omega) \to W^{M+s,p}_0(B(0,2 \, \diam(\Omega)))$ and the Sobolev Embedding Theorem, it suffices to prove the estimate for $f\in C^{M,{\sigma}}(\R^d)$. We will prove only the case $d=1$ leaving to the reader the generalization, which can be obtained by using the one-dimensional result over appropriate paths. In the real case, we define
$$F_t(u):=\frac{f(t)-P^M_u f(t)}{(t-u)^M}$$
for any $u \neq t\in \R$. We want to see that $|F_t(u)|\leq C \norm{f}_{C^{M, {\sigma}}} |u-t|^{\sigma}$ for $t\neq u$.
Note that the $M$-differentiability of $f$ implies that $\lim_{\tau\to t}F_t(\tau)= 0$. Thus, decomposing $P^M_u f(t)=P^{M-1}_u f(t)+\frac{1}{M!}f^{(M)}(u)(t-u)^M$, we have that
\begin{align}\label{eqBreakTaylor}
F_t(u) =\lim_{\tau\to t} F_t(u)-F_t(\tau)
\nonumber	& = \lim_{\tau\to t} \frac{\left(f(t)-P^{M-1}_u f(t)\right)-\left(f(t)-P^{M-1}_\tau f(t)\right)}{(t-u)^M}\\
\nonumber	& \quad +\lim_{\tau\to t} \left( f(t)-P^{M-1}_\tau f(t)\right)\left(\frac{1}{(t-u)^M}-\frac{1}{(t-\tau)^M}\right)\\
			& \quad +\lim_{\tau\to t} \frac{1}{M!} \left(-f^{(M)}(u)+f^{(M)}(\tau)\right)=\circled{I}+\circled{II}+\circled{III}.
\end{align}	
The first term in \rf{eqBreakTaylor} is 
$$\circled{I}=\frac{\left(f(t)-P^{M-1}_u f(t)\right)}{(t-u)^M}$$
and, using the mean value form of the remainder term of the Taylor polynomial, there exists a point $c_1\in (u,t)$ such that 
$$\circled{I}=\frac{f^{(M)}(c_1)}{M!}.$$

The second term in \rf{eqBreakTaylor} is
\begin{align*}
\circled{II}
	& = \lim_{\tau\to t} \left( f(t)-P^{M-1}_\tau f(t)\right)\left(\frac{(t-\tau)^M-(t-u)^M}{(t-u)^M(t-\tau)^M}\right)\\
	& = \lim_{\tau\to t} \left(f(t)-P^{M-1}_\tau f(t)\right)(u-\tau)\left(\sum_{j=1}^M\frac{1}{(t-u)^j (t-\tau)^{M+1-j}}\right)\\
	&=  \lim_{\tau\to t}\frac{u-\tau}{t-u} \left(\sum_{j=1}^M\frac{ f(t)-P^{M-1}_\tau f(t)}{(t-u)^{j-1} (t-\tau)^{M+1-j}}\right)= - \sum_{j=1}^M \lim_{\tau\to t}\frac{ f(t)-P^{M-1}_\tau f(t)}{(t-u)^{j-1} (t-\tau)^{M+1-j}}.
\end{align*}
Applying the Taylor Theorem, only the term $j=1$ has a non-null limit in the last sum, with
\begin{align*}
\circled{II}
	&=  - \frac{f^{(M)}(t)}{M!},
\end{align*}
so 
$$|F_t(u)|\leq \left|\frac{f^{(M)}(c_1)}{M!}-\frac{f^{(M)}(t)}{M!}\right|+\frac{1}{M!} \lim_{\tau\to t}\left|f^{(M)}(u)-f^{(M)}(\tau)\right|\leq \frac{2}{M!} \norm{f}_{C^{M, {\sigma}}} |u-t|^{{\sigma}}.$$
\end{proof}

Recall that in \rf{eqTaylorError} we defined the Taylor error terms
\begin{equation*}
R_{M,j}^{m_3}(z,\xi):=\partial^j h_{m_3}(\xi)-P^{M-j}_z(\partial^j h_{m_3})(\xi)
\end{equation*}
for $M, j,m_3\in \N$ and $z,\xi\in\Omega$. Next we give bounds on the size of these terms.

\begin{lemma}\label{lemKernelBounds}
Consider a real number $p>2$ and naturals $n, m\in \N$ and let $\Omega\subset \C$ be a bounded Lipschitz domain with parameterizations of the boundary in $B^{n+1-1/p}_{p,p}$. Writing ${\sigma_{p}}:=1-\frac2p$, then for $j\leq m$ we have that
\begin{equation}\label{eqBoundRmnjm+1}
|R_{m+n,j}^{m+1}(z,\xi)|\leq C_{\Omega,n,m} |z-\xi|^{m+n-j + {\sigma_{p}}}
\end{equation}
and, if $z_1,z_2,\xi\in\Omega$ with $|z_1-\xi|>\frac32|z_1-z_2|$, then
\begin{equation}\label{eqBoundRmnjm}
|R_{m+n-1,j}^{m}(z_1,\xi)-R_{m+n-1,j}^{m}(z_2,\xi)|\leq C_{\Omega,n,m} |z_1-z_2|^{\sigma_{p}} |z_1-\xi|^{m+n-j-1}.
\end{equation}
\end{lemma}
\begin{proof}
Recall that $\Beurling^k\chi_\Omega \in W^{n,p}({\Omega})$ for every $k$ by Theorem \ref{theoGeometricPGtr2}. Thus,  by \rf{eqDerivativesBeurlings} we have that $\nabla^{m+1}h_{m+1}\in W^{n,p}({\Omega})$ and, since $h_{m+1}$ is bounded in $\Omega$ as well (take absolute values in \rf{eqhm3}), we have that $\partial^j h_{m+1}\in W^{n+m+1-j,p}({\Omega})$ for $0\leq j\leq m+n$. By Lemma \ref{lemTaylor}, it follows that
$$|R_{m+n,j}^{m+1}(z,\xi)|\leq C \norm{\partial^j h_{m+1}}_{W^{m+n-j+1,p}(\Omega)} |z-\xi|^{m+n-j + {\sigma_{p}}}.$$

The second inequality is obtained by the same procedure as \cite[Lemma 7]{MateuOrobitgVerdera}. We quote it here for the sake of completeness. Assume that $z_1,z_2,\xi\in\Omega$ with $|z_1-\xi|>\frac32|z_1-z_2|$. Then
$$R_{m+n-1,j}^{m}(z_1,\xi)-R_{m+n-1,j}^{m}(z_2,\xi)=P_{z_1}^{m+n-1-j}\partial^j h_m(\xi)-P_{z_2}^{m+n-1-j}\partial^j h_m(\xi).$$
But for a natural number $M$ and a function $f\in C^{M,{\sigma_{p}}}(\overline\Omega)$ one has that
\begin{align*}
P_{z_1}^{M}f(\xi)-P_{z_2}^{M}f(\xi)
	& =\sum_{|\vec{i}|\leq M}\frac{D^{\vec{i}}f(z_1)}{\vec{i}!}(\xi-z_1)^{\vec{i}}-\sum_{|\vec{j}|\leq M}\frac{D^{\vec{j}}f(z_2)}{\vec{j}!}(\xi-z_2)^{\vec{j}}.
\end{align*}
Since $(\xi-z_2)^{\vec{j}}=\sum_{\vec{i}\leq\vec{j}}{\vec{j} \choose \vec{i}}(z_1-z_2)^{\vec{j}-\vec{i}}(\xi-z_1)^{\vec{i}}$, one can write
\begin{align*}
P_{z_1}^{M}f(\xi)-P_{z_2}^{M}f(\xi)
	& =\sum_{|\vec{i}|\leq M}\frac{D^{\vec{i}}f(z_1)}{\vec{i}!}(\xi-z_1)^{\vec{i}}-\sum_{|\vec{j}|\leq M}\frac{D^{\vec{j}}f(z_2)}{\vec{j}!}\sum_{\vec{i}\leq\vec{j}}{\vec{j} \choose \vec{i}}(z_1-z_2)^{\vec{j}-\vec{i}}(\xi-z_1)^{\vec{i}}\\
	& =\sum_{|\vec{i}|\leq M}\frac{(\xi-z_1)^{\vec{i}}}{\vec{i}!}\left({D^{\vec{i}}f(z_1)} -\sum_{\substack{|\vec{j}|\leq M\\\vec{i}\leq\vec{j}}}\frac{D^{\vec{j}}f(z_2)}{(\vec{j}-\vec{i})}(z_1-z_2)^{\vec{j}-\vec{i}}\right)\\
	& =\sum_{|\vec{i}|\leq M}\frac{(\xi-z_1)^{\vec{i}}}{\vec{i}!}\left({D^{\vec{i}}f(z_1)} -P_{z_2}^{M-|\vec{i}|}D^{\vec{i}}f(z_1)\right).
\end{align*}
Therefore, Lemma \ref{lemTaylor} may be applied to obtain
\begin{align*}
|P_{z_1}^{M}f(\xi)-P_{z_2}^{M}f(\xi)|
	& \lesssim \sum_{i\leq M} |\xi-z_1|^{i} \norm{f}_{W^{M+1,p}(\Omega)} |z_1-z_2|^{M-i+{\sigma_{p}}}\\
	& \lesssim |\xi-z_1|^{M} |z_1-z_2|^{{\sigma_{p}}}  \norm{f}_{W^{M+1,p}(\Omega)}.
\end{align*}

\end{proof}

We finish this section with a short lemma that we will also use in the proof of Lemma \ref{lemCompactnessDeath}. We say that a function is radial if $f(re^{i\theta})=f(r)$ for every $\theta \in (0,2\pi)$.
\begin{lemma}\label{lemBeurlingBump}
Let $f$ be a radial function in $L^2(\C)$ such that $f|_\DDD\equiv 1$. Then, for every $m \geq 1$,
\begin{equation*}
\Beurling^m f (z)=0\mbox{\quad\quad for }z\in \DDD.
\end{equation*}
\end{lemma}
\begin{proof}
Consider first the characteristic function $\chi_{\DDD}$. Then, 
$$\Beurling \chi_{\DDD} (z) =  - z^{-2} \chi_{\DDD^c},$$
(see \cite[(4.25)]{AstalaIwaniecMartin}).
That is, $\Beurling \chi_{\DDD}$ vanishes in $\DDD$ and coincides with the holomorphic function $z^{-2}$ in $\DDD^c$. 

Next, consider a function defined as $\varphi(z) = z^{j} \bar{z}^{i} \chi_{\DDD^c}(z)$, 
with $j+i\leq -2$ and $i\geq 0$. Let $F(z)=\frac1{i+1}\left(z^{j} \bar{z}^{i+1} - z^{j-i-1} \right) \chi_{\DDD^c}(z)$. It is clear that $F\in W^{1,p}(\C)$ for $p>2$. As a consequence $\Beurling\varphi=\Beurling(\bar\partial F)=\partial F$, that is,
$$\Beurling \varphi (z) = \left(\frac{j}{i+1} z^{j-1}\bar{z}^{i+1} - \frac{j-i-1}{i+1} z^{j-i-2}\right) \chi_{\DDD^c}(z).$$
Thus, the direct sum (understood as a vector space of linear combinations with finitely many non-null coefficients) of the spans of monomials like the ones introduced above 
$$\Phi :=  \bigoplus_{j+i \leq -2: \,i\geq 0}  \langle z^{j} \bar{z}^{i}  \chi_{\DDD^c}(z)\rangle$$
is stable under $\Beurling$, and clearly $\Beurling \chi_\DDD \in \Phi$. The lemma is proven for $f=\chi_\DDD$.

In general, every radial function $f\in L^2(\C)$ such that $f|_\DDD\equiv 1$ can be approximated by finite linear combinations of characteristic functions of concentric circles. The lemma follows combining those facts.\end{proof}

\subsection{Compactness of $\mathcal{R}_m$}\label{secCompactnessDeath}

\begin{proof}[Proof of Lemma \ref{lemCompactnessDeath}]
Recall that we want to prove that $\mathcal{R}_m: f \mapsto \chi_\Omega \Beurling\left(\chi_{\Omega^c} \Beurling^{m-1}(\chi_\Omega f)\right)$ is a compact operator in $W^{n,p}(\Omega)$. 

Since $\mathcal{R}_m f$ is holomorphic in $\Omega$, it is enough to see that $\mathcal{T}_m:=\partial^n \mathcal{R}_m:W^{n,p}(\Omega)\to L^p(\Omega)$ is a compact operator.

Indeed, we have that $\mathcal{R}_m$ is bounded in $W^{n,p}(\Omega)$ by \rf{eqAmIsBounded} and, thus, since the inclusion $W^{n,p}(\Omega)\hookrightarrow W^{n-1,p}(\Omega)$ is compact for any extension domain (see \cite[4.3.2/Remark 1]{TriebelTheory}), we have that $\mathcal{R}_m:W^{n,p}(\Omega)\to W^{n-1,p}(\Omega)$ is compact. That is, given a bounded sequence $\{f_j\}_{j}\subset W^{n,p}(\Omega)$, there exists a subsequence $\{f_{j_k}\}_k$ and a function $g\in W^{n-1,p}(\Omega)$ such that $\mathcal{R}_m f_{j_k}\to g$ in $W^{n-1,p}(\Omega)$. If $\mathcal{T}_m:W^{n,p}(\Omega)\to L^p(\Omega)$ was a compact operator, then there would be a subsubsequence $\{f_{j_{k_i}}\}_i$ and a function $g_n$  such that  $\mathcal{T}_m f_{j_{k_i}}\to g_n$ in $L^p(\Omega)$. It is immediate to see that $g_n$ is the weak derivative $\partial ^n g$ in $\Omega$. Therefore, if $\mathcal{T}_m$ is compact then $\mathcal{R}_m$ is compact as well.

We will prove that  $\mathcal{T}_m$ is compact. Let $f\in W^{n,p}(\Omega)$. Consider a partition of the unity $\{\psi_Q\}_{Q\in\mathcal{W}}$ such that $\supp \, \psi_Q\subset \frac{11}{10}Q$ and $|\nabla^{j}\psi_Q | \lesssim \ell(Q)^{-j}$ for every Whitney cube $Q$ and every $0\leq j\leq n$. 

For every $i\in\N$, let $\Omega_i:=\bigcup_{Q:\ell(Q)>2^{-i}} \supp(\psi_Q)$. We can define a finite partition of the unity $\{\psi_Q^i\}_{Q\in\mathcal{W}}$ such that 
\begin{itemize}
\item If $\ell(Q)>2^{-i}$ then $\psi_Q^i=\psi_Q$.
\item If $\ell(Q)=2^{-i}$ then $\supp \, \psi_Q^i \subset \mathbf{Sh}(Q)$ (see Definition \ref{defShadow}), $|\nabla^{j}\psi_Q^i | \lesssim \ell(Q)^{-j}$ and $\supp\, \psi_Q^i\cap \Omega_i\subset 2Q$.
\item If $\ell(Q)<2^{-i}$ then $\psi_Q^i \equiv 0$. 
\end{itemize}
Indeed, to obtain $\psi_Q^i$ for cubes $Q$ with $\ell(Q)=2^{-i}$, choosing a convenient $\rho>1$ one can consider a partition associated to $\{U_Q\}_{\ell(Q)=2^{-i}}$ for the collection of open sets 
$$U_Q=2Q \cup \{x: \dist(x,Q)<\rho Q,\mbox{ with either } x\in \Omega^c \mbox{ or } x\in S\in\mathcal{W} \mbox{ with } \ell(S)<2^{-i} \} \subset \Sh(Q)$$
 and then multiply each of these  functions times $\left(\chi_\Omega-\sum_{\ell(Q)>2^{-i}}\psi_Q\right)$.

Then, writing $f_Q=\fint_Q f \, dm$ for the mean of $f$ in  $Q$ and $\left(\mathcal{T}_m (f-f_Q) \right)_{Q}=\fint_{Q} \mathcal{T}_m (f-f_Q)\, dm$, we can define 
\begin{align*}
\mathcal{T}_m^i f (z)
	& =\sum_{Q\in \mathcal{W}: \ell(Q)>2^{-i}} \mathcal{T}_m (f) (z) \psi_Q(z)  +\sum_{Q\in \mathcal{W}: \ell(Q)= 2^{-i}} \left(\mathcal{T}_m (f-f_Q) \right)_{Q} \psi_Q^i (z).
\end{align*}

We will prove the following two claims.
\begin{claim}\label{claimCompactApproximation}
For every $i\in\N$, the operator $\mathcal{T}_m^i: W^{n,p}(\Omega) \to L^p(\Omega)$ is compact.
\end{claim}
\begin{claim}\label{claimSmallDifference}
The norm of the error operator $\mathcal{E}^i:=\mathcal{T}_m-\mathcal{T}_m^i: W^{n,p}(\Omega) \to L^p(\Omega)$ tends to zero as $i$ tends to  infinity.
\end{claim}
Then the compactness of $\mathcal{T}_m$ is a well-known consequence of the previous two claims (see \cite[Theorem 4.11]{Schechter}). Putting all together, this proves Lemma \ref{lemCompactnessDeath}.
\end{proof}

\begin{proof}[Proof of Claim \ref{claimCompactApproximation}]
We will prove that the operator $\mathcal{T}_m^i:W^{n,p}(\Omega)\to W^{1,p}(\Omega)$ is bounded. As before, since $\Omega$ is an extension domain, the embedding $W^{1,p}(\Omega)\hookrightarrow L^p(\Omega)$ is compact. Therefore we will deduce the compactness of $\mathcal{T}_m^i:W^{n,p}(\Omega)\to L^p(\Omega)$. Note that the specific value of the operator norm $\norm{\mathcal{T}_m^i}_{W^{n,p}(\Omega)\to W^{1,p}(\Omega)}$ is not important for our argument, since we only care about compactness.

Consider a fixed $i\in\N$ and $f\in W^{n,p}(\Omega)$. For every $z\in \Omega$ and every first order derivative $D$, since $\mathcal{T}_m f$ is analytic on $\Omega$, we can use the Leibniz rule \rf{eqLeibniz} to get
\begin{align*}
D \mathcal{T}_m^i f= \sum_{Q: \ell(Q)>2^{-i}} D\mathcal{T}_m (f) \psi_Q +\sum_{Q:\ell(Q)>2^{-i}} \mathcal{T}_m (f) D \psi_Q +\sum_{ Q:\ell(Q)= 2^{-i}} \left(\mathcal{T}_m (f-f_Q) \right)_{Q} D\psi_Q^i .
\end{align*}
By Jensen's inequality $\left|\mathcal{T}_m (f-f_Q) \right|_{Q}\leq \norm{\mathcal{T}_m (f-f_Q) }_{L^p(Q)}\ell(Q)^{-2/p}$, so 
\begin{align}\label{eqJensenInOperatorCompact}
|\nabla \mathcal{T}_m^i f(z)|
\nonumber	& \leq \sum_{Q:\ell(Q)>2^{-i}}\chi_{\frac{11}{10}Q}(z) |\nabla \mathcal{T}_m f (z)|+ \sum_{Q:\ell(Q)>2^{-i}} |\nabla \psi_Q(z)||\mathcal{T}_m f (z)|\\
	& \quad +\sum_{Q:\ell(Q)=2^{-i}} |\nabla \psi_Q^i(z)| \norm{\mathcal{T}_m (f-f_Q) }_{L^p(Q)}(2^{-i})^{-2/p}.
\end{align}
Using the finite overlapping of the double Whitney cubes and the fact that  $|\nabla \psi_Q^i(z)| \lesssim 2^{i}$ for every Whitney cube $Q$,  we can conclude that 
$$\norm{\nabla \mathcal{T}_m^i  f}_{L^p(\Omega)}^p \lesssim_{i,p}  \norm{\nabla \mathcal{T}_m f }_{L^p(\Omega_i)}^p+\norm{ \mathcal{T}_m f }_{L^{p}(\Omega_i)}^p + \sum_{Q:\ell(Q)=2^{-i}}\left(\norm{\mathcal{T}_m f }_{L^p(Q)}^p + |f_Q|^p\norm{\mathcal{T}_m 1}_{L^p(Q)}^p \right).$$
By the Sobolev Embedding Theorem 
\begin{align}\label{eqSobolevfQ}
|f_Q|\leq \norm{f}_{L^\infty(\Omega)}\lesssim_{\Omega,p} \norm{f}_{W^{1,p}(\Omega)}.
\end{align}
Thus, since $\mathcal{T}_m:W^{n,p}(\Omega)\to L^p(\Omega)$ is bounded, we have that
\begin{equation}\label{eqBoundTBreak}
\norm{\nabla \mathcal{T}_m^i f}_{L^p(\Omega)} \lesssim_{p,i,\Omega} \norm{\nabla \mathcal{T}_m f }_{L^p(\Omega_i)}+ \norm{f}_{W^{n,p}(\Omega)}.
\end{equation}

To see that $\norm{\nabla \mathcal{T}_m f }_{L^p(\Omega_i)}\lesssim_i\norm{f}_{W^{n,p}(\Omega)}$, note that $\nabla\mathcal{T}_mf=\nabla \partial^{n} \Beurling\left(\chi_{\Omega^c} \Beurling^{m-1}(\chi_\Omega f)\right)$. We have that $\Beurling^{m-1} :L^p(\Omega)\to L^p(\Omega^c)$ is bounded trivially, and for $z\in \Omega_i$ and $g\in L^p$ supported in $\Omega^c$ we have that
$$\left|\nabla\partial^{n} \Beurling g(z)\right|\lesssim \int_{|z-w|>2^{-i}} \frac{1}{|z-w|^{n+3}} g(w)\, dm(w).$$
This is the convolution of $g$ with an $L^1$ kernel, so Young's inequality \rf{eqYoung} tells us that
$$\norm{\nabla \partial^{n} \Beurling g}_{L^p(\Omega_i)}\leq C_i \norm{g}_{L^p},$$
proving that
\begin{equation}\label{eqBoundTPartBad}
\norm{\nabla \mathcal{T}_mf}_{L^p(\Omega_i)}\lesssim_i \norm{\Beurling^{m-1}(\chi_\Omega f)}_{L^p(\Omega^c)}\lesssim \norm{f}_{L^{p}(\Omega)}\lesssim \norm{f}_{W^{n,p}(\Omega)}.
\end{equation}

 Combining \rf{eqBoundTBreak} and \rf{eqBoundTPartBad}, we have seen that $\norm{\nabla\mathcal{T}_m^i f}_{L^p(\Omega)}\lesssim \norm{f}_{W^{n,p}(\Omega)}$. The reader can use Jensen's inequality as in \rf{eqJensenInOperatorCompact} to check that $\norm{\mathcal{T}_m^i f}_{L^p(\Omega)}\lesssim \norm{f}_{W^{n,p}(\Omega)}$ as well. This, proves that the operator $\mathcal{T}_m^i:W^{n,p}(\Omega)\to W^{1,p}(\Omega)$ is bounded and, therefore, composing with the compact inclusion, the operator $\mathcal{T}_m^i:W^{n,p}(\Omega)\to L^p(\Omega)$ is compact.
\end{proof}

\begin{proof}[Proof of Claim \ref{claimSmallDifference}]
We want to see that the error operator 
$$\mathcal{E}^i  =\mathcal{T}_m-\mathcal{T}_m^i $$
satisfies that $\norm{\mathcal{E}^i}_{W^{n,p}(\Omega) \to L^p(\Omega)}$ tends to zero as $i$ tends to  infinity.

 Recall that $\Omega_i=\bigcup_{Q:\ell(Q)>2^{-i}} \supp(\psi_Q)$. We define the modified error operator $\mathcal{E}^i_0$ acting in $f\in W^{n,p}(\Omega)$ as
$$\mathcal{E}^i_0 f (z):= \chi_{\Omega\setminus\Omega_{i-1}}(z) \sum_{Q: \ell(Q)= 2^{-i}} \sum_{\substack{S: \ell(S)\leq 2^{-i}\\S\subset \mathbf{Sh}(Q)}} \left|\mathcal{T}_m(f-f_S)(z) - \left(\mathcal{T}_m (f-f_Q) \right)_{Q} \right|  \chi_{2S}(z)$$
for every $z\in\Omega$.
The first step will be proving that
\begin{equation}\label{eqObjectiveError}
\norm{\mathcal{E}^i f}_{L^p(\Omega)} \lesssim \norm{\mathcal{E}^i_0 f}_{L^p(\Omega)} + C_i \norm{f}_{W^{1,p}(\Omega)},
\end{equation}
with $C_i\xrightarrow{i\to\infty} 0$.

Note that $\mathcal{T}_m 1=\mathcal{T}_m \chi_\Omega$ because $\mathcal{T}_m f=\partial^n\chi_\Omega \Beurling\left(\chi_{\Omega^c} \Beurling^{m-1}(\chi_\Omega f)\right)$. Let us write 
\begin{align*}
\mathcal{T}_m f (z)
	& = \sum_{S\in\mathcal{W}:\ell(S)> 2^{-i}}  \mathcal{T}_m (f) (z) \psi_S (z)+ \sum_{S\in\mathcal{W}:\ell(S)\leq 2^{-i}}\left( f_S \mathcal{T}_m (1) (z)  + \mathcal{T}_m (f-f_S) (z) \right) \psi_S (z)
\end{align*}
for $z\in\Omega$.  Recall that
\begin{align*}
\mathcal{T}_m^i f (z)
	& = \sum_{Q\in \mathcal{W}: \ell(Q)>2^{-i}} \mathcal{T}_m (f) (z) \psi_Q(z)  +\sum_{Q\in \mathcal{W}: \ell(Q)= 2^{-i}} \left(\mathcal{T}_m (f-f_Q) \right)_{Q} \psi_Q^i (z).
\end{align*}
Thus, for the error operator $\mathcal{E}^i$ we have the expression
\begin{align}\label{eqBreakError}
\mathcal{E}^i f (z)
\nonumber	& =\mathcal{T}_mf(z)-\mathcal{T}_m^if(z)= \sum_{S:\ell(S)\leq 2^{-i}} f_S \mathcal{T}_m (1) (z) \psi_S(z) \\
\nonumber	& \quad + \left(\sum_{S: \ell(S)\leq 2^{-i}} \mathcal{T}_m(f-f_S)(z) \psi_S(z) -\sum_{Q: \ell(Q)= 2^{-i}}  \left(\mathcal{T}_m (f-f_Q) \right)_{Q}  \psi_Q^i (z)\right)\\
	& = \mathcal{E}^i_1f(z)+\mathcal{E}^i_2f(z).
\end{align}

The first part is easy to bound using again \rf{eqSobolevfQ}. Indeed, we have that
\begin{align}\label{eqBoundE0}
\norm{\mathcal{E}^i_1f}_{L^p(\Omega)}^p 
 \lesssim_p \sum_{S:\ell(S)\leq 2^{-i}} |f_S|^p \norm{\mathcal{T}_m (1)}_{L^p(11/10S)}^p
 \lesssim_\Omega\norm{f}_{W^{1,p}(\Omega)}^p \norm{\mathcal{T}_m (1)}_{L^p(\Omega\setminus \Omega_{i-1})}^p,
\end{align}
where $\norm{\mathcal{T}_m (1)}_{L^p(\Omega\setminus \Omega_i)}^p \xrightarrow{i\to \infty} 0$.

To control $\mathcal{E}^i_2f$ in \rf{eqBreakError}, note that every $z \in \Omega$ satisfies
\begin{equation}\label{eqSumPsi}
\sum_{S:\ell(S)\leq 2^{-i}} \psi_S(z)=\sum_{Q:\ell(Q)=2^{-i}}\psi_Q^i(z)\leq 1,
\end{equation}
with equality when $z\notin \bigcup_{\ell(Q)>2^{-i}} \supp(\psi_Q)$, that is, when $z\in \Omega \setminus\Omega_i$. 
In this case
\begin{align}\label{eqEi1nonlocal}
\mathcal{E}^i_2 f(z) 
\nonumber	& = \sum_{S: \ell(S)\leq 2^{-i}} \mathcal{T}_m(f-f_S)(z) \psi_S(z)\sum_{Q: \ell(Q)= 2^{-i}} \psi_Q^i (z)\\
\nonumber	& \quad -\sum_{Q: \ell(Q)= 2^{-i}}  \left(\mathcal{T}_m (f-f_Q) \right)_{Q}  \psi_Q^i (z) \sum_{S: \ell(S)\leq 2^{-i}} \psi_S(z) \\
	& = \sum_{Q: \ell(Q)= 2^{-i}} \sum_{S: \ell(S)\leq 2^{-i}} \left(\mathcal{T}_m(f-f_S)(z) - \left(\mathcal{T}_m (f-f_Q) \right)_{Q} \right) \psi_S(z)\psi_Q^i (z) .
\end{align}
If, instead, $z\in \Omega_i=\bigcup_{Q:\ell(Q)>2^{-i}} \supp(\psi_Q)$ then there is a cube $S_0$ with $z\in \supp(\psi_{S_0})$ and $\ell(S_0)\geq 2^{-i+1}$. Therefore, any other cube $S$  with $\psi_S(z)\neq 0$ must be a neighbor of $S_0$ and, therefore, it has side-length $\ell(S)\geq 2^{-i}$  (see Section \ref{secNotation}). Therefore,
\begin{align*}
\mathcal{E}^i_2 f(z) 
	& = \sum_{S: \ell(S)= 2^{-i}} \mathcal{T}_m(f-f_S)(z) \psi_S(z) -\sum_{Q: \ell(Q)= 2^{-i}}  \left(\mathcal{T}_m (f-f_Q) \right)_{Q}  \psi_Q^i (z) \\
	& = \sum_{Q: \ell(Q)= 2^{-i}}\left( \mathcal{T}_m(f-f_Q)(z) \psi_Q(z) -\left(\mathcal{T}_m (f-f_Q) \right)_{Q}  \psi_Q^i (z)\right).
\end{align*}
Adding and subtracting $\mathcal{T}_m(f-f_Q)(z)\psi_Q^i (z)$ at each term of this sum, we get 
\begin{align}\label{eqEi1local}
\mathcal{E}^i_2 f(z) 
\nonumber	& = \sum_{Q: \ell(Q)= 2^{-i}}  \mathcal{T}_m(f-f_Q)(z) \left(\psi_Q(z)-\psi_Q^i (z) \right) \\
	& \quad + \sum_{Q: \ell(Q) = 2^{-i}}\left(\mathcal{T}_m(f-f_Q)(z) - \left(\mathcal{T}_m (f-f_Q)\right)_{Q}\right) \psi_Q^i(z).
\end{align}

Summing up, by \rf{eqEi1nonlocal} and \rf{eqEi1local} we have that
\begin{align*}
\mathcal{E}^i_2 f (z)
	& = 	\chi_{\Omega\setminus\Omega_i}(z) \sum_{Q: \ell(Q)= 2^{-i}} \sum_{S: \ell(S)\leq 2^{-i}} \left(\mathcal{T}_m(f-f_S)(z) - \left(\mathcal{T}_m (f-f_Q) \right)_{Q} \right) \psi_S(z)\psi_Q^i (z) \\
	& \quad +	\chi_{\Omega_i\setminus\Omega_{i-1}}(z) \sum_{Q: \ell(Q)= 2^{-i}} \left( \mathcal{T}_m(f-f_Q)(z) - \left(\mathcal{T}_m (f-f_Q) \right)_{Q} \right) \psi_Q^i (z)\\
	& \quad +	\chi_{\Omega_i\setminus\Omega_{i-1}}(z) \sum_{Q: \ell(Q)= 2^{-i}}  \mathcal{T}_m(f-f_Q)(z) \left(\psi_Q(z)-\psi_Q^i (z)\right).
\end{align*}
Since every cube $Q$ with $\ell(Q)= 2^{-i}$ satisfies that $\supp \, \psi_Q^i \subset \mathbf{Sh}(Q)$ and $\supp\, \psi_Q^i\cap \Omega_i\subset 2Q$, we get that
\begin{align}\label{eqBoundEi2well}
|\mathcal{E}^i_2 f (z)|
	& \lesssim \chi_{\Omega\setminus\Omega_{i-1}}(z) \sum_{Q: \ell(Q)= 2^{-i}} \sum_{\substack{S: \ell(S)\leq 2^{-i}\\S\subset \mathbf{Sh}(Q)}} \left|\mathcal{T}_m(f-f_S)(z) - \left(\mathcal{T}_m (f-f_Q) \right)_{Q} \right| \chi_{2S}(z)\\
\nonumber	& \quad  + \chi_{\Omega_i\setminus\Omega_{i-1}}(z) \left|\sum_{Q: \ell(Q)= 2^{-i}}  \mathcal{T}_m(f-f_Q)(z) \left(\psi_Q(z) - \psi_Q^i (z)\right)\right|.
\end{align}

The first term coincides with $\mathcal{E}^i_0 f$. For the last term, just note that whenever $z\in \Omega_i\setminus\Omega_{i-1}$, using the first equality in \rf{eqSumPsi} we have that 
$$\sum_{Q: \ell(Q)= 2^{-i}}  \mathcal{T}_m(f)(z) \left(\psi_Q^i (z)-\psi_Q(z)\right)=  \mathcal{T}_m(f)(z) \left(\sum_{Q: \ell(Q)= 2^{-i}}\psi_Q^i (z)-\sum_{Q: \ell(Q)= 2^{-i}}\psi_Q(z)\right)\equiv 0.$$
Thus, for $z\in \Omega_i\setminus\Omega_{i-1}$
\begin{align*}
\sum_{Q: \ell(Q)= 2^{-i}}  \mathcal{T}_m(f-f_Q)(z) \left(\psi_Q^i (z)-\psi_Q(z)\right)
	& =\sum_{Q: \ell(Q)= 2^{-i}}  -\mathcal{T}_m(f_Q)(z) \left(\psi_Q^i (z)-\psi_Q(z)\right),
\end{align*}
which can be bounded as $\mathcal{E}^i_1$ in \rf{eqBoundE0} by noting the finite overlap of supports of $\psi_Q^i$ for cubes $Q$ of size $2^{-i}$. This fact, together with \rf{eqBreakError}, \rf{eqBoundE0} and \rf{eqBoundEi2well} settles \rf{eqObjectiveError}, that is, 
\begin{equation*}
\norm{\mathcal{E}^i f}_{L^p(\Omega)} \lesssim \norm{\mathcal{E}^i_0 f}_{L^p(\Omega)} + C_{i,\Omega,n,p} \norm{f}_{W^{1,p}(\Omega)},
\end{equation*}
with $C_{i,\Omega,n,p}\xrightarrow{i\to\infty} 0$.

Recall that we defined the modified error term
$$\mathcal{E}^i_0 f (z)= \chi_{\Omega\setminus\Omega_{i-1}}(z) \sum_{Q: \ell(Q)= 2^{-i}} \sum_{\substack{S: \ell(S)\leq 2^{-i}\\S\subset \mathbf{Sh}(Q)}} \left|\mathcal{T}_m(f-f_S)(z) - \left(\mathcal{T}_m (f-f_Q) \right)_{Q} \right|  \chi_{2S}(z).$$ 
Next we prove that $\norm{\mathcal{E}^i_0 f}_{L^p(\Omega)} \lesssim  C_i \norm{f}_{W^{1,p}(\Omega)}$, with $C_i\xrightarrow{i\to\infty} 0$.

Arguing by duality, we have that
\begin{align}\label{eqSumBefore}
\norm{\mathcal{E}^i_0 f }_{L^p}
 = \sup_{g: \norm{g}_{p'}=1}  \int_{\Omega\setminus \Omega_{i-1}} \sum_{\substack{Q: \ell(Q)= 2^{-i}\\S: \ell(S)\leq 2^{-i}\\S\subset \mathbf{Sh}(Q)}} \left|\mathcal{T}_m(f-f_S)(z) - \left(\mathcal{T}_m (f-f_Q) \right)_{Q} \right|  \chi_{2S}(z)\, |g(z)| \, dm(z) .
\end{align}
First note for every pair of Whitney cubes $Q$ and $S$ with $S\subset \mathbf{Sh}(Q)$ and every point $z$, using an admissible chain $[S,Q)=[S,Q]\setminus \{Q\}$ we get that 
\begin{align*}
\mathcal{T}_m(f-f_S)(z) - \left(\mathcal{T}_m (f-f_Q) \right)_{Q} 
	& = \mathcal{T}_m(f-f_S)(z)-\left(\mathcal{T}_m (f-f_S) \right)_{S} \\
	& \quad + \sum_{P\in [S,Q)} \left(\mathcal{T}_m (f-f_P) \right)_{P} - \left(\mathcal{T}_m (f-f_{\mathcal{N}(P)}) \right)_{\mathcal{N}(P)},
\end{align*}
where $\mathcal{N}(P)$ stands for the ``next'' cube in the chain $[S,Q]$ (see Remark \ref{remChain}).
Note that the shadows of cubes of fixed side-length have finite overlapping since $|\mathbf{Sh}(Q)|\approx |Q|$ and, therefore, every Whitney cube $S$  appears less than $C$ times in the right-hand side of \rf{eqSumBefore}. Thus,
\begin{align}\label{eqBreakEi0fAlmost}
\norm{\mathcal{E}^i_0 f }_{L^p}
 			& \lesssim \sup_{g: \norm{g}_{p'}=1} \bigg(  \sum_{S: \ell(S)\leq 2^{-i}} \int_{2S}\left|\mathcal{T}_m(f-f_S)(z) - \left(\mathcal{T}_m (f-f_S) \right)_{S} \right| |g(z)| \, dm(z) \\
\nonumber 	& \quad +   \sum_{\substack{Q: \ell(Q)= 2^{-i}\\S: \ell(S)\leq 2^{-i}\\S\subset \mathbf{Sh}(Q)}}  \sum_{P\in [S,Q)} \left|\left(\mathcal{T}_m (f-f_P) \right)_{P} - \left(\mathcal{T}_m (f-f_{\mathcal{N}(P)})\right)_{\mathcal{N}(P)}\right|  \int_{2S} |g(z)| \, dm(z) \bigg).
\end{align}

All the cubes $P\in [S,Q]$ with  $S\in \SH(Q)$, satisfy that $\ell(P)\lesssim \Dist(Q,S)\approx \ell(Q)$ by Remark \ref{remChain} and Definition \ref{defShadow}. If we assume that $\ell(Q)=2^{-i}$ this implies that $\ell(P)\leq C2^{-i}$. Moreover, we have that 
\begin{align}\label{eqTheMeanIsCool}
\left|\left(\mathcal{T}_m (f-f_P) \right)_{P} - \left(\mathcal{T}_m (f-f_{\mathcal{N}(P)})\right)_{\mathcal{N}(P)}\right| 
& \leq \sum_{L\cap 2P\neq \emptyset}\fint_{P}\left|\mathcal{T}_m (f-f_{P})(z)-\left(\mathcal{T}_m (f-f_L) \right)_{L} \right|\, dm(z).
\end{align}
Finally, we observe that $P\in [S,Q]$ with $S\subset \mathbf{Sh}(Q)$ imply that $\Dist(P,S)\leq C\ell(P)$. Indeed, if $P\in [S,S_Q]$ then this comes from \rf{eqDPQClose}
and, if $P\in[Q_S,Q]$  by \rf{eqDPQClose} we have that $\ell(P)\approx\Dist(P,Q)\geq \ell(Q)$ and by \rf{eqDPQFar} $\ell(Q)\approx\Dist(Q,S)\approx\Dist(P,S)$. Thus, for a fixed $P$ with $\ell(P)\leq C2^{-i}$ and $g\in L^{p'}$, we have that
\begin{equation}\label{eqSumQSP}
\sum_{\substack{Q: \ell(Q)= 2^{-i}\\S:S\subset \mathbf{Sh}(Q)\\ P\in [S,Q]}} \int_{2S} |g(z)| \, dm(z) \lesssim C\sum_{S:\Dist(P,S)\leq C\ell(P)}  \int_{2S} |g(z)| \, dm(z) \lesssim \ell(P)^d \inf_{P} Mg. 
\end{equation}
Note that we again used that every cube $S$ appears less than $C$ times in the left-hand side. By \rf{eqBreakEi0fAlmost}, \rf{eqTheMeanIsCool} and applying \rf{eqSumQSP} after reordering, we get that
\begin{align*}
\norm{\mathcal{E}^i_0 f }_{L^p}
	& \lesssim \sup_{\norm{g}_{p'}=1}  \sum_{\substack{S:\ell(S)\leq C 2^{-i}\\L\cap 2S\neq \emptyset}}  \int_{2S}  \left|\left(\mathcal{T}_m(f-f_S)(z)-\left(\mathcal{T}_m (f-f_L) \right)_{L}\right) \left(|g(z)| +Mg(z)\right)\right| \, dm(z).
\end{align*}
Since $\norm{Mg}_{L^{p'}}\lesssim \norm{g}_{L^{p'}}\leq 1$,  we have that
\begin{equation*}
\norm{\mathcal{E}^i_0 f }_{L^p}
	 \lesssim \sup_{\norm{g}_{p'}=1}  \sum_{(S,L)\in \mathcal{W}_0} \int_{2S}  \left|\mathcal{T}_m(f-f_S)(z)-\left(\mathcal{T}_m (f-f_L) \right)_{L}\right|\left| g(z) \right| \, dm(z),
\end{equation*}
where $\mathcal{W}_0=\{(S,L):\ell(S)\leq C2^{-i}\mbox{ and } 2S\cap L\neq \emptyset\}$.

 For every cube $Q$, let $\varphi_Q$ be a radial bump function with $\chi_{10Q}\leq \varphi_Q\leq \chi_{20Q}$ and the usual bounds in their derivatives. Now we use these bump functions to separate the local and the non-local parts. In the local part we can neglect the cancellation and use the triangle inequality, but in the non-local part the smoothness of a certain kernel will be crucial, so we write
\begin{align}\label{eqBreakEi1}
\norm{\mathcal{E}^i_0 f }_{L^p}
\nonumber	& \lesssim\sup_{\norm{g}_{p'}=1}  \sum_{S:\ell(S)\leq C2^{-i}} \int_{2S}  \left|\mathcal{T}_m[(f-f_S)\varphi_S](z)\right| |g(z)| \, dm(z)\\
\nonumber	& \quad +\sup_{\norm{g}_{p'}=1} \sum_{(S,L)\in\mathcal{W}_0} \fint_{L}  \left|\mathcal{T}_m[(f-f_L)\varphi_S](\xi)\right|\, dm(\xi) \int_{2S} |g|\, dm\\
\nonumber	& \quad +\sup_{\norm{g}_{p'}=1}\sum_{(S,L)\in\mathcal{W}_0} \int_{2S}  \Big|\mathcal{T}_m[(f-f_S)(1 - \varphi_S)](z)-(\mathcal{T}_m [(f-f_L)(1 - \varphi_S)])_{L}\Big| |g(z)| dm(z) \\
			& = \squared{7'}+\squared{7''}+\squared{8}.
\end{align}
Since $\fint_{2S}|g|\, dm\lesssim \inf_{7S}Mg$ and $2L\subset 7S$, we have that 
 $$ \squared{7''} \lesssim \sup_{\norm{g}_{p'}=1} \sum_{(S,L)\in\mathcal{W}_0} \int_{2L}  \left|\mathcal{T}_m[(f-f_L)\varphi_S](\xi)\right| Mg(\xi) \, dm(\xi)  =\squared{7}.$$
Note that the inequality $|g|\leq Mg$ (which is valid almost everywhere for $g$ in $L^1_{loc}$) imply that $\squared{7'}\leq \squared{7}$ as well. 

First we take a look at $\squared{7}$. For any pair of neighbor Whitney cubes $S$ and $L$ and $z\in 2L$, using the definition of weak derivative and Fubini's Theorem we find that
\begin{align*}
\mathcal{T}_m[(f-f_L)\varphi_S](z)
	& =c_n\int_{\Omega^c}\frac{1}{(z-w)^{n+2}}\int_{20S} \frac{(\overline{w-\xi})^{m-1}}{(w-\xi)^{m+1}}(f(\xi)-f_L)\varphi_S(\xi)\, dm(\xi)\, dm(w)\\
	& =c_{n,m}\int_{\Omega^c}\frac{1}{(z-w)^{n+2}}\int_{20S} \frac{(\overline{w-\xi})^{m}}{(w-\xi)^{m+1}}\overline{\partial}[(f-f_L)\varphi_S](\xi)\, dm(\xi)\, dm(w)\\
	& =c_{n,m}\int_{20S} \left(\int_{\Omega^c} \frac{(\overline{w-\xi})^{m}}{(w-\xi)^{m+1}(z-w)^{n+2}} dm(w)\right) \overline{\partial}[(f-f_L)\varphi_S](\xi)\, dm(\xi).
\end{align*}
In the right-hand side above, we have that $\xi,z\in\Omega$. Therefore, we can use Green's Theorem in the integral on $\Omega^c$ and then \rf{eqKernelVectorM} to get 
\begin{align*}
\mathcal{T}_m[(f-f_L)\varphi_S](z)
	& =c_{n,m}\int_{20S} \left(\int_{\partial\Omega} \frac{(\overline{w-\xi})^{m+1}}{(w-\xi)^{m+1}(z-w)^{n+2}} dw\right) \overline{\partial}[(f-f_L)\varphi_S](\xi)\, dm(\xi)\\
	& =c_{n,m} \int_{20S} K_{\vec{m}_0}(z,\xi) \overline{\partial}[(f-f_L)\varphi_S](\xi)\, dm(\xi),
\end{align*}
where $\vec{m}_0:=(2+n,m+1,m+1)$.

Using Proposition \ref{propoKernelExpression} we have that
$$K_{\vec{m}_0}(z,\xi)=c_{m,n}{\partial^{n}}\Beurling\chi_\Omega (z) \frac{(\overline{\xi-z})^{m}}{(\xi-z)^{m+1}}+\sum_{j\leq m} \frac{c_{m,n,j} R_{m+n,j}^{m+1}(z,\xi)}{(\xi-z)^{m+n+2-j}}.$$
The first part is ${\partial^{n}}\Beurling\chi_\Omega (z)$ times the kernel of the operator $T^{(-m-1,m)}$ (see Theorem \ref{theoGeometricPGtr2}).
For the second part, we have that by Lemma \ref{lemKernelBounds} 
$$\frac{|R_{m+n,j}^{m+1}(z,\xi)|}{|\xi-z|^{m+n+2-j}}\lesssim\frac{1}{|\xi-z|^{2-{\sigma_{p}}}},$$
where ${\sigma_{p}}=1-\frac2p$. Thus, 
\begin{align}\label{eqBreak7}
\squared{7}
			& =\sup_{\norm{g}_{p'}=1}  \sum_{(S,L)\in\mathcal{W}_0} \int_{2L}  \left|\mathcal{T}_m[(f-f_L)\varphi_S](z)\right| Mg(z) \, dm(z)\\
\nonumber	& \lesssim\sup_{\norm{g}_{p'}=1}  \sum_{(S,L)\in\mathcal{W}_0} \int_{2L}  \left|{\partial^{n}}\Beurling\chi_\Omega (z) T^{(-m-1,m)}\left(\overline\partial[(f-f_L)\varphi_S]\right)(z)\right| Mg(z) \, dm(z)\\
\nonumber	& \quad +\sup_{\norm{g}_{p'}=1}  \sum_{(S,L)\in\mathcal{W}_0} \int_{2L}   \int_{20S} \frac{ \left|\overline{\partial}[(f-f_L)\varphi_S](\xi)\right|}{|\xi-z|^{2-{\sigma_{p}}}} \, dm(\xi) Mg(z) \, dm(z) = \squared{7.1}+\squared{7.2}.
\end{align}

In the first sum we use that in $W^{1,p}(\C)$ we have the identity $T^{(-m-1,m)}\circ \overline\partial= \overline\partial\circ T^{(-m-1,m)}=c_m \Beurling^m$ and, therefore, $T^{(-m-1,m)}\overline\partial[(f-f_L)\varphi_S]=c_m \Beurling^m [(f-f_L)\varphi_S]\in W^{1,p}\subset L^\infty$, so
\begin{align*}
\squared{7.1}
	& \lesssim\sup_{\norm{g}_{p'}=1}  \sum_{(S,L)\in\mathcal{W}_0} \int_{2L}  \left|{\partial^{n}}\Beurling\chi_\Omega (z) \right| Mg(z) \, dm(z)\norm{ \Beurling^m[(f-f_L)\varphi_S]}_{L^\infty}\\
	& \lesssim\sup_{\norm{g}_{p'}=1}  \sum_{(S,L)\in\mathcal{W}_0} \norm{\partial^n \Beurling \chi_\Omega}_{L^p(2L)}\norm{Mg}_{L^{p'}(2L)} \norm{\Beurling^m[(f-f_L)\varphi_S]}_{W^{1,p}(\C)}.
\end{align*}
By the boundedness of $\Beurling^m$ in $W^{1,p}(\C)$ we have that 
$$\norm{\Beurling^m[(f-f_L)\varphi_S]}_{W^{1,p}(\C)}\lesssim \norm{(f-f_L)\varphi_S}_{W^{1,p}(20S)}.$$
Moreover, the Poincar\'e inequality \rf{eqPoincare} allows us to deduce that
\begin{equation}\label{eqPoincare20S}
\norm{(f-f_L)\varphi_S}_{W^{1,p}(20S)} \lesssim  \norm{\nabla f}_{L^p(20S)}.
\end{equation}

On the other hand, there is a certain $i_0$ such that for $\ell(S)\leq C 2^{-i}$ and $L\cap 2S\neq \emptyset$, we have that $S,2L\subset \Omega\setminus \Omega_{i-i_0}$, and
$$\norm{\partial^n \Beurling \chi_\Omega}_{L^p(2L)}\leq \norm{\partial^n \Beurling \chi_\Omega}_{L^p(\Omega\setminus \Omega_{i-i_0})}.$$
Thus, by the H\"older inequality and the boundedness of the maximal operator in $L^{p'}$ we have that
\begin{align}\label{eq71End}
\squared{7.1}
\nonumber	& \lesssim \norm{\partial^n \Beurling \chi_\Omega}_{L^p(\Omega \setminus\Omega_{i-i_0})} \sup_{\norm{g}_{p'}=1}  \sum_{(S,L)\in\mathcal{W}_0}\norm{Mg}_{L^{p'}(2L)} \norm{\nabla f}_{L^p(20S)} \\
			& \leq C_{\Omega,i} \norm{\nabla f}_{L^p(\Omega)} \sup_{\norm{g}_{p'}=1} \norm{Mg}_{L^{p'}}\lesssim_{p} C_{\Omega,i} \norm{\nabla f}_{L^p(\Omega)},
\end{align}
with $C_{\Omega,i}\xrightarrow{i\to\infty}0$.

To bound the term \squared{7.2} in \rf{eqBreak7}, note that given two neighbor cubes $S$ and $L$ and a point $z\in 2L$, integrating on dyadic annuli we have that
$$\int_{20S} \frac{ \left|\overline{\partial}[(f-f_L)\varphi_S](\xi)\right|}{|\xi-z|^{2-{\sigma_{p}}}} \, dm(\xi) \lesssim M \left(\overline{\partial}[(f-f_L)\varphi_S]\right)(z)\ell(S)^{\sigma_{p}}.$$
Thus, 
\begin{align*}
\squared{7.2}
	& \lesssim\sup_{\norm{g}_{p'}=1}  \sum_{(S,L)\in\mathcal{W}_0} \int_{2L}  M \left(\overline{\partial}[(f-f_L)\varphi_S]\right)(z)\ell(S)^{\sigma_{p}} Mg(z) \, dm(z)\\
	& \lesssim 2^{-i{\sigma_{p}}} \sup_{\norm{g}_{p'}=1}  \sum_{(S,L)\in\mathcal{W}_0} \norm{M \left(\overline{\partial}[(f-f_L)\varphi_S]\right)}_{L^p(\Omega)}\norm{Mg}_{L^{p'}(2L)}
\end{align*}
and, by the boundedness of the maximal operator, \rf{eqPoincare20S} and the H\"older inequality, we get
\begin{align}\label{eq72End}
\squared{7.2}
	& \lesssim 2^{-i{\sigma_{p}}} \sup_{\norm{g}_{p'}=1}  \sum_{(S,L)\in\mathcal{W}_0} \norm{\overline{\partial}[(f-f_L)\varphi_S]}_{L^p(20S)}\norm{Mg}_{L^{p'}(2L)} \lesssim 2^{-i{\sigma_{p}}} \norm{\nabla f}_{L^p(\Omega)}.
\end{align}
By \rf{eqBreak7}, \rf{eq71End} and \rf{eq72End}, we have that
\begin{align}\label{eq7End}
\squared{7}
	& \lesssim C_{\Omega,i} \norm{\nabla f}_{L^p(\Omega)},
\end{align}
with $C_{\Omega,i}\xrightarrow{i\to\infty}0$.

Back to \rf{eqBreakEi1} it remains to bound
\begin{align*}
\squared{8}
	& = \sup_{\norm{g}_{p'}=1} \sum_{(S,L)\in\mathcal{W}_0} \int_{2S}  \left|\mathcal{T}_m[(f-f_S)(1 - \varphi_S)](z)-\left(\mathcal{T}_m [(f-f_L)(1- \varphi_S)] \right)_{L}\right| |g(z)| \, dm(z).
\end{align*}
Fix $g\geq 0$ such that $\norm{g}_{p'}=1$. Then we will prove that
$$\squared{8g}\leq C_{\Omega,i}\norm{f}_{W^{1,p}(\Omega)},$$
with $C_{\Omega,i}\xrightarrow{i\to\infty}0$, where
\begin{align*}
\squared{8g}
	& := \sum_{(S,L)\in\mathcal{W}_0} \int_{2S}  \fint_L \left|\mathcal{T}_m[(f-f_S)(1 - \varphi_S)](z)- \mathcal{T}_m [(f-f_L)(1 - \varphi_S)] (\zeta) \right| \,dm(\zeta) g(z) \, dm(z).
\end{align*}

First, we add and subtract $\mathcal{T}_m[(f-f_L)(1 - \varphi_S)](z)$ in each term of the last sum to get
\begin{align*}
\squared{8g}
	& \leq  \sum_{(S,L)\in\mathcal{W}_0} \int_{2S} \left|\mathcal{T}_m[(f_L-f_S)(1 - \varphi_S)](z) \right| \fint_L\,dm(\zeta) g(z) \, dm(z)\\
	& \quad+  \sum_{(S,L)\in\mathcal{W}_0} \int_{2S}  \fint_L \left|\mathcal{T}_m[(f-f_L)(1 - \varphi_S)](z)- \mathcal{T}_m [(f-f_L)(1 - \varphi_S)] (\zeta) \right| \,dm(\zeta) g(z) \, dm(z).
\end{align*}
For a given $z\in\Omega$,
$$\int_{\Omega^c}\int_{\Omega}\frac{|f(\xi)-f_L|}{|z-w|^{n+2}|w-\xi|^2} \,dm(\xi)\, dm(w)\lesssim \norm{f}_{L^\infty}\int_{\Omega^c}\frac{|\log(\dist(w,\Omega))|+O(1)}{|z-w|^{n+2}} \, dm(w),$$
which is finite since $\Omega$ is a Lipschitz domain (hint: compare the last integral above with the length of the boundary $\mathcal{H}^1(\partial\Omega)$ times the integral $\int_0^1|\log(t)| \, dt$).
Thus, we can use Fubini's Theorem and then Green's Theorem to state that
\begin{align*}
\mathcal{T}_m[(f-f_L)(1 - \varphi_S)](z)
	& = c_n\int_{\Omega^c}\frac{1}{(z-w)^{n+2}}\int_{\Omega} \frac{(\overline{w-\xi})^{m-1}}{(w-\xi)^{m+1}}(f(\xi)-f_L)(1-\varphi_S(\xi))\, dm(\xi)\, dm(w)\\
 	& =c_{n,m}\int_{\Omega} \left(\int_{\partial\Omega} \frac{(\overline{w-\xi})^{m}}{(w-\xi)^{m+1}(z-w)^{n+2}} dw\right) [(f-f_L)(1-\varphi_S)](\xi)\, dm(\xi)\\
	& =c_{n,m} \int_{\Omega} K_{\vec{m}_1}(z,\xi) [(f-f_L)(1-\varphi_S)](\xi)\, dm(\xi),
\end{align*}
where $\vec{m}_1:=(2+n,m+1,m)$. 
Arguing analogously, 
\begin{align*}
\mathcal{T}_m[(f_L-f_S)(1 - \varphi_S)](z)
	& =c_{n,m} (f_L-f_S)\int_{\Omega\setminus 10S} K_{\vec{m}_1}(z,\xi) (1-\varphi_S)(\xi)\, dm(\xi).
\end{align*}
Thus, we get that
\begin{align}\label{eqBreak8AB}
\squared{8g}
\nonumber	& \lesssim  \sum_{(S,L)\in\mathcal{W}_0}  |f_L-f_S| \int_{2S} \left|\int_{\Omega\setminus 10S} K_{\vec{m}_1}(z,\xi) [(1-\varphi_S)](\xi)\, dm(\xi) \right| g(z) \, dm(z)\\
\nonumber	& \quad+  \sum_{(S,L)\in\mathcal{W}_0} \int_{2S}  \fint_L \left|\int_{\Omega}(K_{\vec{m}_1}(z,\xi)- K_{\vec{m}_1}(\zeta,\xi)) [(f-f_L)(1-\varphi_S)](\xi)\, dm(\xi) \right| dm(\zeta) g(z) dm(z)\\
	& =\squared{8.1}+\squared{8.2}.
\end{align}

Recall that Proposition \ref{propoKernelExpression} states that for $z\in 2S$ and $\xi\in\Omega$, 
\begin{equation}\label{eqKernelExplicit}
K_{\vec{m}_1}(z,\xi)=c_{m,n}{\partial^{n}}\Beurling\chi_\Omega (z) \frac{(\overline{\xi-z})^{m-1}}{(\xi-z)^{m+1}}+\sum_{j\leq m} \frac{c_{m,n,j} R_{m+n-1,j}^{m}(z,\xi)}{(\xi-z)^{m+n+2-j}}
\end{equation}
and, for any $z,\xi\in\Omega$, by \rf{eqBoundRmnjm+1} we have that 
\begin{equation}\label{eqBoundLocalRmn1j}
\left|\frac{R_{m+n-1,j}^{m}(z,\xi)}{ (z-\xi)^{m+n+2-j }}\right|\leq C_{\Omega,n,m} \frac{1}{|z-\xi|^{ 3-{\sigma_{p}}}},
\end{equation}
where ${\sigma_{p}}=1-\frac2p$. Thus, by \rf{eqBoundRmnjm} and the identity $\frac{1}{a^j}-\frac{1}{b^j}=\frac{(b-a)\sum_{i=0}^{j-1}a^ib^{j-1-i}}{a^jb^j}$, when $z,\zeta \in 5S$ and $\xi\in\Omega\setminus 20S$ we have that
\begin{align}\label{eqBoundNonlocalRmnjm}
&\left| \frac{R_{m+n-1,j}^{m}(z,\xi)}{(\xi-z)^{m+n+2-j}}-\frac{R_{m+n-1,j}^{m}(\zeta,\xi)}{(\xi-\zeta)^{m+n+2-j}}\right|
\nonumber	\leq \left|R_{m+n-1,j}^{m}(z,\xi)\left(\frac{1}{(\xi-z)^{m+n+2-j}}-\frac{1}{(\xi-\zeta)^{m+n+2-j}}\right)\right|\\
	& \quad +\left|\frac{R_{m+n-1,j}^{m}(z,\xi)-R_{m+n-1,j}^{m}(\zeta,\xi)}{(\xi-\zeta)^{m+n+2-j}}\right|\lesssim_{\Omega,n,m} \frac{|z-\zeta| }{|z-\xi|^{4-{\sigma_{p}}}}+\frac{|z-\zeta|^{\sigma_{p}} }{|z-\xi|^{3}}\lesssim\frac{|z-\zeta|^{\sigma_{p}} }{|z-\xi|^{3}}.
\end{align}
Then, using that $\dist(2S,\supp(1-\varphi_S))>0$, we have that  $\int_{\Omega}\frac{(\overline{\xi-z})^{m-1}}{(\xi-z)^{m+1}} [(1-\varphi_S)](\xi)\, dm(\xi) =c_m\Beurling^m_\Omega [(1-\varphi_S)](z)$ for $z\in 2S$ and, by \rf{eqBreak8AB}, \rf{eqKernelExplicit} and \rf{eqBoundLocalRmn1j} we get that
\begin{align}\label{eqBreak81}
\squared{8.1}
\nonumber	& \lesssim  \sum_{(S,L)\in\mathcal{W}_0} |f_L-f_S|\int_{2S}  \left| \partial^n\Beurling \chi_\Omega(z)\Beurling^m_\Omega [(1-\varphi_S)](z) \right| g(z) \, dm(z)\\
\nonumber	& \quad +  \sum_{(S,L)\in\mathcal{W}_0} |f_L-f_S|\int_{2S} \int_{\Omega\setminus 10S} \frac{1}{|z-\xi|^{ 3-{\sigma_{p}}}} \, dm(\xi) g(z) \, dm(z)\\
	& =\squared{8.1.1}+ \squared{8.1.2}.
\end{align}
By the same token, using \rf{eqBreak8AB}, \rf{eqKernelExplicit} and \rf{eqBoundNonlocalRmnjm} we get 
\begin{align}\label{eqBreak82}
\squared{8.2}
\nonumber	& \lesssim \sum_{(S,L)\in\mathcal{W}_0} \int_{2S}  \fint_L  \left| \partial^n\Beurling \chi_\Omega(z)\Beurling^m_\Omega [(f-f_L)(1-\varphi_S)](z) \right| \,dm(\zeta) g(z) \, dm(z)\\
\nonumber	& \quad + \sum_{(S,L)\in\mathcal{W}_0} \int_{2S}  \fint_L  \left| \partial^n\Beurling \chi_\Omega(\zeta)\Beurling^m_\Omega [(f-f_L)(1-\varphi_S)](\zeta) \right|\,dm(\zeta) g(z) \, dm(z)\\
\nonumber	& \quad + \sum_{(S,L)\in\mathcal{W}_0} \int_{2S}  \fint_L\int_{\Omega\setminus 10S} \frac{|z-\zeta|^{\sigma_{p}}}{|z-\xi|^{ 3}} |f(\xi)-f_L|\, dm(\xi) \,dm(\zeta) g(z) \, dm(z)\\
	& =\squared{8.2.1}+ \squared{8.2.2}+\squared{8.2.3}.
\end{align}

We begin by the first term in the right-hand side of \rf{eqBreak81}, that is, 
\begin{equation*}
\squared{8.1.1}=\sum_{(S,L)\in\mathcal{W}_0} |f_L-f_S|\int_{2S}  \left| \partial^n\Beurling \chi_\Omega(z)\Beurling^m_\Omega [(1-\varphi_S)](z) \right| g(z) \, dm(z).
\end{equation*}
By the Poincar\'e and the Jensen inequalities, we have that
\begin{equation}\label{eqPoincareDoubleMean}
|f_L-f_S|\leq \frac{1}{\ell(L)^2}\int_{L} |f(\xi)-f_S|\, dm(\xi)\lesssim \frac{\ell(L)}{\ell(L)^2}\norm{\nabla f}_{L^1(5S)}\lesssim \ell(S)^{1-\frac{2}{p}} \norm{\nabla f}_{L^p(5S)}.
\end{equation}
On the other hand, applying Lemma \ref{lemBeurlingBump} to $\varphi_S$ (conveniently rescaled and translated), we have that
$\Beurling^m \varphi_S (z)=0$ for $z\in 2S$. Therefore, using \rf{eqPoincareDoubleMean} we have that
\begin{align}\label{eq811End}
\squared{8.1.1}
	&  \lesssim\norm{\nabla f}_{L^p(\Omega)} \sum_{S:\ell(S)\leq C2^{-i}} \ell(S)^{1-\frac{2}{p}}   \int_{2S}  \left| \partial^n\Beurling \chi_\Omega(z)\Beurling^m_\Omega \chi_\Omega (z) \right| g(z) \, dm(z)\\
\nonumber	& \lesssim 2^{-i\left(1-\frac{2}{p}\right)}\norm{\nabla f}_{L^p(\Omega)} \norm{g}_{L^{p'}(\Omega)}\norm{\partial^n\Beurling \chi_\Omega}_{L^p(\Omega)}\norm{\Beurling^m_\Omega \chi_\Omega}_{L^\infty(\Omega)}\lesssim_\Omega 2^{-i(1-\frac{2}{p})}\norm{\nabla f}_{L^p(\Omega)}.
\end{align}

Let us recall that the second term in the right-hand side of \rf{eqBreak81} is
\begin{align*}
\squared{8.1.2}
	& = \sum_{(S,L)\in\mathcal{W}_0} |f_L-f_S|\int_{2S} \int_{\Omega\setminus 10S} \frac{1}{|z-\xi|^{ 3-{\sigma_{p}}}} \, dm(\xi) g(z) \, dm(z)
\end{align*}
and, by \rf{eqPoincareDoubleMean},
\begin{align*}
\squared{8.1.2}
	& \lesssim \sum_{S:\ell(S)\leq C2^{-i}} \ell(S)^{1-\frac{2}{p}} \norm{\nabla f}_{L^p(5S)}  \frac{1}{\ell(S)^{1-{\sigma_{p}}}} \norm{g}_{L^1(2S)} \\
	& \lesssim \sum_{S:\ell(S)\leq C2^{-i}} \ell(S)^{{\sigma_{p}}-\frac{2}{p}+\frac{2}{p}} \norm{\nabla f}_{L^p(5S)}  \norm{g}_{L^{p'}(2S)}. 
\end{align*}
By H\"older's inequality, 
\begin{align*}
\squared{8.1.2}
	& \lesssim 2^{-i{\sigma_{p}}}  \norm{\nabla f}_{L^p(\Omega)}  \norm{g}_{L^{p'}(\Omega)}= 2^{-i{\sigma_{p}}}  \norm{\nabla f}_{L^p(\Omega)}.
\end{align*}
Using this fact together with \rf{eqBreak81} and \rf{eq811End}, we have that
\begin{align}\label{eq81End}
\squared{8.1}
	& \lesssim C_{\Omega,i}  \norm{\nabla f}_{L^p(\Omega)},
\end{align}
with $C_{\Omega,i}\xrightarrow{i\to\infty}0$.

Let us focus now on the first term in the right-hand side of \rf{eqBreak82}, that is,
\begin{align}\label{eq821}
\squared{8.2.1}
	& =\sum_{(S,L)\in\mathcal{W}_0} \int_{2S}   \left| \partial^n\Beurling \chi_\Omega(z)\right|\left|\Beurling^m_\Omega [(f-f_L)(1-\varphi_S)](z) \right| g(z) \, dm(z)\\
\nonumber	& \lesssim \sum_{(S,L)\in\mathcal{W}_0}\norm{g}_{L^{p'}(2S)} \norm{\partial^{n}\Beurling\chi_\Omega}_{L^p(2S)} \norm{\Beurling^{m}_\Omega[(f-f_L)(1-\varphi_S)]}_{L^\infty(2S)}.
\end{align}
 By the Sobolev Embedding Theorem and the boundedness of $\Beurling^m_\Omega$ in $W^{1,p}(\Omega)$ (granted by Theorem \ref{theoGeometricPGtr2}) we have that
\begin{align*}
\norm{\Beurling^{m}_\Omega[(f-f_L)(1-\varphi_S)]}_{L^\infty(\Omega)}
	& \leq \norm{\Beurling^{m}_\Omega[(f-f_L)(1-\varphi_S)]}_{W^{1,p}(\Omega)} \lesssim \norm{(f-f_L)(1-\varphi_S)}_{W^{1,p}(\Omega)}
\end{align*}
and, using Leibniz' rule, Poincar\'e's inequality and the Sobolev embedding Theorem, we get
\begin{align*}
\norm{\Beurling^{m}_\Omega[(f-f_L)(1-\varphi_S)]}_{L^\infty(\Omega)}
	& \leq \norm{\nabla f}_{L^p(\Omega)}+\frac{1}{\ell(S)}\norm{f-f_L}_{L^p(20S)}+\norm{f-f_L}_{L^p(\Omega)}\\
	& \lesssim_\Omega \norm{\nabla f}_{L^p(\Omega)}+\norm{\nabla f}_{L^p(20S)}+\norm{f}_{L^p(\Omega)}+\norm{f}_{L^\infty}\lesssim\norm{f}_{W^{1,p}(\Omega)}.
\end{align*}
Thus, by H\"older's inequality we have that
\begin{align}\label{eq821End}
\squared{8.2.1}
	& \lesssim  \norm{f}_{W^{1,p}(\Omega)}  \norm{g}_{L^{p'}(\Omega)} \norm{\partial^{n}\Beurling\chi_\Omega}_{L^p(\Omega\setminus\Omega_{i-i_0})}=  \norm{f}_{W^{1,p}(\Omega)}  \norm{\partial^{n}\Beurling\chi_\Omega}_{L^p(\Omega\setminus\Omega_{i-i_0})}.
\end{align}
Note that $\norm{\partial^{n}\Beurling\chi_\Omega}_{L^p(\Omega\setminus\Omega_{i-i_0})}\xrightarrow{i\to \infty}0 $.

The second term in \rf{eqBreak82}, that is,  
\begin{equation*}
\squared{8.2.2} = \sum_{(S,L)\in\mathcal{W}_0} \frac{1}{\ell(L)^2} \int_L  \left| \partial^n\Beurling \chi_\Omega(\zeta)\Beurling^m_\Omega [(f-f_L)(1-\varphi_S)](\zeta) \right|\,dm(\zeta)  \int_{2S} g(z) \, dm(z),
\end{equation*}
follows the same pattern. Since $S$ and $L$ in the sum above are neighbors, they have comparable side-length, and for $\zeta\in L$ we have that $ \int_{2S} g(z) \, dm(z)\lesssim \ell(L)^2 Mg(\zeta).$ Therefore, 
\begin{align*}
\squared{8.2.2} 
	& \lesssim \sum_{(S,L)\in\mathcal{W}_0}  \int_L  \left| \partial^n\Beurling \chi_\Omega(\zeta)\Beurling^m_\Omega [(f-f_L)(1-\varphi_S)](\zeta) \right|  Mg(\zeta)\,dm(\zeta)\\
	& \lesssim \sum_{{S:\ell(S)\leq C2^{-i}}}\norm{Mg}_{L^{p'}(5S)} \norm{\partial^{n}\Beurling\chi_\Omega}_{L^p(5S)} \norm{\Beurling^{m}_\Omega[(f-f_L)(1-\varphi_S)]}_{L^\infty(5S)}.
\end{align*}
The last expression coincides with the right-hand side of \rf{eq821} changing $g$ by $Mg$ and $2S$ by $5S$. Arguing analogously to that case, we get that 
\begin{align}\label{eq822End}
\squared{8.2.2}
	& \lesssim  \norm{f}_{W^{1,p}(\Omega)}  \norm{Mg}_{L^{p'}(\Omega)} \norm{\partial^{n}\Beurling\chi_\Omega}_{L^p(\Omega\setminus\Omega_{i-i_0})}\lesssim  \norm{f}_{W^{1,p}(\Omega)}  \norm{\partial^{n}\Beurling\chi_\Omega}_{L^p(\Omega\setminus\Omega_{i-i_0})}.
\end{align}

Finally, we consider 
$$\squared{8.2.3}=\sum_{(S,L)\in\mathcal{W}_0} \int_{2S}  \fint_L\int_{\Omega\setminus 10S} \frac{|z-\zeta|^{\sigma_{p}}}{|z-\xi|^{ 3}} |f(\xi)-f_L|\, dm(\xi) \,dm(\zeta) g(z) \, dm(z).$$
Note that for $z,\zeta\in 5S$ we have that $|z-\zeta|\lesssim \ell(S)$. Separating $\Omega\setminus 10S$ in Whitney cubes we get
\begin{align*}
\squared{8.2.3}
	& \lesssim \sum_{(S,L)\in\mathcal{W}_0} \int_{2S}g(z)dm(z)   \sum_{P\in\mathcal{W}}\frac{\ell(S)^{\sigma_{p}} }{\Dist(S,P)^{3}}\norm{f-f_L}_{L^1(P)}.
\end{align*}
Using the chain connecting two cubes $P$ and $L$, by \rf{eqChain} we get that
$$\norm{f-f_L}_{L^1(P)}\lesssim \sum_{Q\in[P,L]}\norm{\nabla f}_{L^1(5Q)}\frac{\ell(P)^2}{\ell(Q)}.$$
Thus,
\begin{align*}
\squared{8.2.3}
	& \lesssim 2^{-i{\sigma_{p}}} \sum_{L,P\in\mathcal{W}}\sum_{Q\in[P,L]}\frac{\ell(P)^2 \norm{\nabla f}_{L^1(5Q)} \norm{g}_{L^1(7L)}}{\ell(Q)\Dist(L,P)^{3}}.
\end{align*}
By Lemma \ref{lemTwoFoldedFunctionGuay} (with $\rho=1$), we get
\begin{equation}\label{eq823End}
\squared{8.2.3}  \lesssim 2^{-i{\sigma_{p}}}\norm{\nabla f}_{L^p(\Omega)},
\end{equation}
and Claim \ref{claimSmallDifference} is proven. Indeed, by \rf{eqBreak82}, \rf{eq821End}, \rf{eq822End} and \rf{eq823End}, we have that
\begin{equation*}
\squared{8.2} \lesssim C_{\Omega,i}  \norm{\nabla f}_{L^p(\Omega)}.
\end{equation*}
 This fact combined with \rf{eqBreak8AB} and \rf{eq81End} prove that
\begin{equation*}
\squared{8}\leq \sup_{\norm{g}_{p'}=1} \squared{8g} \lesssim\ C_{\Omega,i}  \norm{\nabla f}_{L^p(\Omega)}
\end{equation*}
and, together with \rf{eqObjectiveError}, \rf{eqBreakEi1} and \rf{eq7End}, gives
\begin{equation*}
\norm{\mathcal{E}^i f}_{L^p(\Omega)} \lesssim\ C_{\Omega,i}  \norm{f}_{W^{1,p}(\Omega)},
\end{equation*}
with $C_{\Omega,i}\xrightarrow{i\to\infty}0$.
\end{proof}

\renewcommand{\abstractname}{Acknowledgements}
\begin{abstract}
The author was funded by the European Research
Council under the European Union's Seventh Framework Programme (FP7/2007-2013) /
ERC Grant agreement 320501. Also,  partially supported by grants 2014-SGR-75 (Generalitat de Catalunya), MTM-2010-16232 and MTM-2013-44304-P (Spanish government) and by a FI-DGR grant from the Generalitat de Catalunya, (2014FI-B2 00107).

The author would like to thank Xavier Tolsa's advice on his Ph.D. thesis, which gave rise to this work, Cruz, Mateu, Orobitg and Verdera for their advice and interest, and the editor and the referee for their patient work and valuable comments.
\end{abstract}

\bibliography{../../../bibtex/Llibres}
\end{document}